\documentclass[aos,preprint]{imsart}

\RequirePackage[OT1]{fontenc}
\RequirePackage{amsthm,amsmath}
\RequirePackage[numbers]{natbib}
\RequirePackage[colorlinks,citecolor=blue,urlcolor=blue]{hyperref}
\usepackage{amsmath, amsfonts, amssymb, amsthm}
\usepackage{mathrsfs}

\usepackage{graphicx}
\usepackage{subfigure}

\usepackage{changes}

\startlocaldefs
\theoremstyle{plain}
\newtheorem{theorem}{Theorem}
\newtheorem*{thm*}{Theorem}

\newtheorem{proposition}{Proposition}
\newtheorem{lemma}{Lemma}

\newtheorem*{lem*}{Lemma}
\newtheorem{definition}{Definition}
\newtheorem*{defi*}{Definition}

\newtheorem*{rques*}{Remarks}

\endlocaldefs

\newcommand{\bt}{\boldsymbol \theta}

\begin{document}

\begin{frontmatter}

\title{Sharp optimal recovery in the two Component Gaussian Mixture Model}
\runtitle{ Sharp optimal recovery in the Gaussian Mixture Model}

\begin{aug}
\author{\fnms{Mohamed} \snm{Ndaoud}\ead[label=e1]{ndaoud@usc.edu}}


\affiliation{University of Southern California}

\address{Department of Mathematics\\
University of Southern California\\
Los Angeles, CA 90089 \\
\printead{e1}}

\end{aug}

\begin{abstract}
    This paper studies the problem of clustering in the two-component Gaussian mixture model where the centers are separated by $2\Delta$ for some $\Delta>0$.
    We characterize the exact phase transition threshold, given by
$$
\bar{\Delta}_n^{2} = \sigma^{2}\left(1 + \sqrt{1+\frac{2p}{n\log{n}}} \right)\log{n},
$$
such that perfect recovery of the communities is possible with high probability if $\Delta\ge(1+\varepsilon)\bar \Delta_n$, and impossible if $\Delta\le (1-\varepsilon)\bar \Delta_n$
for any constant $\varepsilon>0$. This implies an elbow effect at a critical dimension $p^{*}=n\log{n}$.

    We present a non-asymptotic lower bound for the corresponding minimax Hamming risk improving on existing results. It is, to our knowledge, the first lower bound capturing the right dependence on $p$. We also propose an optimal, efficient and adaptive procedure that is minimax rate optimal. The rate optimality is moreover sharp in the asymptotics when the sample size goes to infinity. Our procedure is based on a variant of Lloyd's iterations initialized by a spectral method; a popular clustering algorithm widely used by practitioners. Numerical studies confirm our theoretical findings.
\end{abstract}





\end{frontmatter}
 \section{Introduction}
 The problems of supervised or unsupervised clustering have gained huge interest in the machine learning literature. In particular, many clustering algorithms are known to achieve good empirical results. A useful model to study and compare these algorithms is the Gaussian mixture model. In this model, we assume that the data are attributed to different centers and that we only have access to observations corrupted by Gaussian noise. For this specific model, one can consider the problem of estimation of the centers, see, e.g., \cite{klusowski},\cite{villar} or the problem of detecting the communities, see, e.g., \cite{zhou},\cite{fei2018hidden},\cite{giraud},\cite{royer}. This paper focuses on community detection.
\subsection{The Gaussian Mixture Model}

We observe $n$ independent random vectors $Y_{1},\dots,Y_{n}\in \mathbf{R}^{p}$. We assume that there exist two unknown vectors $\bt \in \mathbf{R}^{p}$ and $\eta \in \{-1,1\}^{n}$, such that, for all $i=1, \dots,n$,
\begin{equation}\label{model:freq}
Y_{i} = \bt \eta_{i} + \sigma \xi_{i}, 
\end{equation}
where $\sigma>0$, $\xi_{1},\dots,\xi_{n}$ are standard Gaussian random vectors and $\eta_{i}$ is the $i$th component of $\eta$. We denote by $Y$ (respectively, $W$) the matrix with columns $Y_{1},\dots,Y_{n}$ (respectively, $\sigma \xi_{1},\dots,\sigma \xi_{n}$). Model \eqref{model:freq} can be written in matrix form
$$
Y = \bt \eta^{\top} +  W.
$$
We denote by $\mathbf{P}_{(\bt,\eta)}$ the distribution of $Y$ in model \eqref{model:freq} and by $\mathbf{E}_{(\bt,\eta)}$ the corresponding expectation. We assume that $(\bt,\eta)$ belongs to the set
$$
\Omega_{\Delta} = \{ \bt \in \mathbf{R}^{p}: \| \bt\| \geq \Delta\} \times \{-1,1 \}^{n},
$$
where $\Delta>0$ is a given constant. The value $\Delta$ characterizes the separation between the clusters and equivalently the strength of the signal.

In this paper, we study the problem of recovering the communities, that is, of estimating the vector $\eta$.
As estimators of $\eta$, we consider any measurable functions
$\hat\eta=\hat\eta(Y_1,\dots,Y_n)$ of $(Y_1,\dots,Y_n)$
taking values in $\{-1,1\}^n$. 
We characterize the loss of a given $\hat\eta$ by the Hamming distance between $\hat\eta$
and $\eta$, that is, by the number of positions at which $\hat
\eta$ and $\eta$ differ:
\[
| \hat\eta-\eta | := \sum_{j=1}^n
| \hat\eta_j-\eta_j | = 2\sum
_{j=1}^n \mathbf{1}(\hat\eta_j\ne
\eta_j).
\]
Here, $\hat\eta_j$ and $ \eta_j$ are the $j$th
components of $\hat\eta$ and $ \eta$, respectively. Since for community detection it is enough to determine $\eta$ up to a sign change, in what follows we use the loss defined by
$$
r(\hat{\eta},\eta):= \underset{\nu\in\{-1,1\}}{\min}|\hat{\eta}-\nu\eta|.
$$
The expected loss of $\hat{\eta}$ is defined as
$\mathbf{E}_{(\bt,\eta)} r(\hat{\eta},\eta) $.

In the rest of the paper, we will always denote by $\eta$ the true vector to estimate, while $\hat{\eta}$ will denote the corresponding estimator. We consider the following minimax risk
\begin{equation}\label{def:minimax}
    \Psi_{\Delta}:= \underset{\tilde{\eta}}{\inf} \underset{(\bt,\eta) \in \Omega_{\Delta}}{\sup} \frac{1}{n}\mathbf{E}_{(\bt,\eta)}r(\tilde{\eta},\eta),
\end{equation}
where $\underset{\tilde{\eta}}{\inf}$ denotes the infimum over all estimators $\tilde{\eta}$ valued in $\{-1,1\}^{n}$.
A simple lower bound for the risk $\Psi_{\Delta}$ is given by (cf. Proposition \ref{eq:lower_trivial} below):
\begin{equation}\label{eq:lower_trivial}
    \Psi_{\Delta} \geq \frac{c}{1+\Delta/\sigma}e^{-\frac{\Delta^{2}}{2\sigma^{2}}}
\end{equation}
for some $c>0$. The above lower bound is dimension independent. Inspecting its proof, one may notice that this bound is attained at the oracle $\eta^{*}$ given by
$$ \eta_{i}^{*} = \text{sign}\left(  Y_{i}^{\top} \bt \right). $$
This oracle assumes a prior knowledge of $\bt$. It turns out that for $p \geq n$, there exists a regime where the lower bound \eqref{eq:lower_trivial} is not optimal, as pointed by \cite{giraud}. The intuitive explanation is that for $p$ larger than $n$, although the vector $\bt$ is hard to estimate, perfect recovery of the communities is still possible. To the best of our knowledge, there are no lower bounds for $\Psi_{\Delta}$ that capture the issue of estimating $\bt$. This is one of the main questions addressed in the present paper.

{\bf Notation.}
In the rest of this paper we use the following notation. For given sequences $a_{n}$ and $b_n$, we write that $a_{n} = \mathcal{O}(b_{n})$ (respectively, $a_{n} = \Omega(b_n)$) when $a_{n} \leq c b_{n}$ (respectively, $a_{n} \geq c b_{n}$) for some absolute constant $c>0$. We write $a_{n} \asymp b_{n} $ when $a_{n} = \mathcal{O}(b_{n})$ and $a_{n}=\Omega(b_{n})$. For $x,y\in \mathbf{R}^{p}$, we denote by $x^{\top}y$ the Euclidean scalar product, by $\|x\|$ the corresponding norm of $x$ and by $\text{sign}(x)$ the vector of signs of the components of $x$. For $x,y \in \mathbf{R}$, we denote by $x\vee y$ (respectively, $x \wedge y$) the maximum (respectively, minimum) value between $x$ and $y$. To any matrix $M \in \mathbf{R}^{n\times p}$, we denote by $\|M\|_{op}$ its operator norm with respect to the $L^{2}$-norm , by $M^{\top}$ its transpose and by $\text{Tr}(M)$ its trace in case $p=n$. Further, $\mathbf{I}_{n}$ denotes the identity matrix of dimension $n$ and $\mathbf{1}(.)$ denotes the indicator function. We denote by $\Phi^{{\sf c}}(.)$ the complementary cumulative distribution function of the standard Gaussian random variable $z$ i.e., $\forall t \in \mathbf{R}, \Phi^{{\sf c}}(t)=\mathbf{P}(z > t)$. We denote by $c$ and $C$ positive constants that may vary from line to line.

We assume that $p,\sigma$ and $\Delta$ depend on $n$ and the asymptotic results correspond to the limit as $n \to \infty$. All proofs are deferred to the Appendix.

\subsection{Related literature}
The present work can be related to two parallel lines of work. 
\subsubsection{Community detection in the sub-Gaussian mixture model}

        \noindent To our knowledge, \cite{zhou} were  the first to present statistical guarantees for community detection in the sub-Gaussian mixture model using the well-known Lloyd's algorithm, cf. \cite{lloyds}. The results of \cite{zhou} require a better initialization than a random guess in addition to the condition
\begin{equation}\label{cond:zhou:1}
\Delta^{2} =\Omega\left(\sigma^{2}\left( 1 \vee \frac{p}{n}\right)\right),
\end{equation}
in order to achieve \textit{almost full recovery recovery} and
\begin{equation}\label{cond:zhou:2}
\Delta^{2} =\Omega\left(\sigma^{2}\log{n}\left( 1 \vee \frac{p}{n}\right)\right),
\end{equation}
in order to achieve \textit{exact recovery}. The notions of almost full and exact recovery are defined in Section \ref{sec:sharp}.
More recently, \cite{royer} and \cite{giraud} have shown that conditions \eqref{cond:zhou:1} and \eqref{cond:zhou:2} are not optimal in high dimension i.e. for $n=o(p)$. In particular, \cite{giraud} study an SDP relaxation of the K-means criterion that achieves \textit{almost full  recovery} under a milder condition
\begin{equation}\label{cond:giraud:part}
    \Delta^{2} =\Omega\left(\sigma^{2}\left( 1 \vee \sqrt{\frac{p}{n}}\right)\right),
\end{equation}
and \textit{exact recovery} under the condition
\begin{equation}\label{cond:giraud:exact}
    \Delta^{2} =\Omega\left(\sigma^{2}\left( \log{n} \vee \sqrt{\frac{p\log{n}}{n}}\right)\right).
\end{equation}
To the best of our knowledge, conditions \eqref{cond:giraud:part} and \eqref{cond:giraud:exact} are the mildest in the literature, but no matching necessary conditions are known so far. \cite{giraud} provide insightful heuristics about optimality of these conditions. In the supervised setting, where all labels are known similar conditions seem necessary to achieve either almost full or exact recovery. It is still not clear whether optimal conditions in supervised mixture learning are also optimal in the unsupervised setting. 

Another difference between the previous papers is in computational aspects. While, in \cite{giraud}, an SDP relaxation is proposed, a faster algorithm based on Lloyd's iterations is developed in \cite{zhou}. It remains unclear whether we can achieve almost full (respectively, exact) recovery under condition \eqref{cond:giraud:part} (respectively, \eqref{cond:giraud:exact}) through faster methods than SDP relaxations, for instance, through Lloyd's iterations.
 
 \cite{zhou} suggest to initialize Lloyd's algorithm using a spectral method. It would be interesting to investigate whether Lloyd's algorithm initialized by a spectral method, in the same spirit as in \cite{vempala}, can achieve optimal performance in the more general setting where $p$ is allowed to be larger than $n$. 

In this paper, we shed some light on these issues. Specifically, we address the following questions.
\begin{itemize}
    \item Are conditions \eqref{cond:giraud:part} and \eqref{cond:giraud:exact} necessary for both almost full and exact recovery? 
    \item Are optimal requirements similar in both supervised and unsupervised settings? 
    \item Can we achieve results similar to \cite{giraud} using a faster algorithm? 
    \item In case the answer to previous questions is positive, can we achieve the same results adaptively to all parameters?
\end{itemize}
\subsubsection{Community detection in the Stochastic Block Model (SBM)}

\noindent The Stochastic Block Model, cf. \cite{holland}, is probably the most popular framework for node clustering. This model with two communities can be seen as a particular case of model \eqref{model:freq} when both the signal and the noise are symmetric matrices. A non symmetric variant of SBM is the Bipartite SBM, cf. \cite{feldman}. Unlike the case of sub-Gaussian mixtures where most results in the literature are non-asymptotic, results on almost full or exact recovery for the SBM and its variants are mostly asymptotic and focus on sharp phase transitions. \cite{abbe} poses an open question on whether it is possible to characterize sharp phase transitions in other related problems, for instance, in the Gaussian mixture model. 

The first polynomial method achieving exact recovery in the SBM with two communities is due to \cite{bandeira}. The algorithm splits the initial sample into two independent samples. A black-box algorithm is used on the first sample for almost full recovery, then a local improvement is applied on the second sample. As stated in \cite{bandeira}, it is not clear whether algorithms achieving almost full recovery can be used to achieve exact recovery. It remains interesting to understand whether similar results can be achieved through direct algorithms ideally without the splitting step. 

For the Bipartite SBM, sufficient computational conditions for exact recovery are presented in \cite{feldman}, \cite{perkins}. While the sharp phase transition for the problem of detection is fully answered in \cite{perkins}, it is still not clear whether the condition they require, for exact recovery, is optimal. More interestingly, the sufficient condition for exact recovery is different for $p$ of the same order as $n$ and for $p$ larger than $n^{2}$ for instance. This shows a kind of phase transition with respect to $p$, where for some critical dimension $p^{*}$ the hardness of the problem changes.

We resume potential connections between our work and these recent developments in the following questions.
\begin{itemize}
    \item Is it possible to characterize a sharp phase transition for exact recovery in model \eqref{model:freq}?
    \item Are algorithms achieving almost full recovery useful in order to achieve exact recovery in the Gaussian mixture model?
    \item Is there a critical dimension $p^{*}$ that separates different regimes of hardness in the problem of exact recovery?
\end{itemize}

\subsection{Main contribution}

This work provides a sharp analysis of almost full and exact recovery in the two component Gaussian mixture model.
As detailed in Section \ref{sec:sharp}, we show the existence of a sharp phase transition for exact recovery in the Gaussian mixture model as $\Delta$ approaches the
threshold $\bar{\Delta}_{n}$ given by
$$ \bar{\Delta}^{2}_{n} = \sigma^{2}  \left( 1 + \sqrt{1 + \frac{2p}{n\log{n}}}\right)\log{n}. $$ 
For any fixed $\epsilon>0$, exact recovery of the communities is possible if $\Delta\ge (1+\epsilon) \bar\Delta_n$
and impossible if $\Delta\le (1-\epsilon) \bar\Delta_n$.
In particular, this phase transition gives rise to two different regimes around a critical dimension $p^{*} = n\log{n}$, showing that the hardness of exact recovery depends on whether $p$ is larger or smaller than $p^{*}$:
\begin{itemize}
    \item 
If $p=o(p^{*}) $, then $
\bar{\Delta}_{n}=(1+o(1))\sigma\sqrt{2\log{n}}$. In this case, the phase transition threshold for exact recovery is the same as if $\bt$ were known.
\item
 While if $p^{*}=o(p)$, then $\bar{\Delta}_{n}=(1+o(1))\sigma\left(\frac{2p\log{n}}{n}\right)^{1/4}$. This new condition takes into account the hardness of estimation, and $p^{*}$ can be interpreted as a phase transition with respect to the hardness of estimation of $\bt$.
\end{itemize}

The above findings are formalized as non-asymptotic lower bounds for the minimax risk $\Psi_{\Delta}$ and matching upper bounds through a variant of Lloyd's iterations initialized by a spectral method. To do so, we define a key quantity $\mathbf{r}_{n}$ that turns out to be the right signal-to-noise ratio (SNR) of the problem:
\begin{equation}\label{def:SNR}
\mathbf{r}_{n} =  \frac{\Delta^{2}/\sigma^{2}}{\sqrt{\Delta^{2}/\sigma^{2}+p/n}}.
\end{equation}
This SNR is strictly smaller than the "naive" one $\Delta/\sigma$, cf. \eqref{eq:lower_trivial}. In particular, it states that the hardness of the problem depends on the dimension $p$. Among other results, we prove that for some $c_{1},c_{2},C_{1},C_{2}>0$, we have
$$
 C_{1}e^{-c_{1}\mathbf{r}^{2}_{n}} \leq \Psi_{\Delta} \leq C_{2}e^{-c_{2}\mathbf{r}^{2}_{n}}.
$$
Moreover, we give a sharp characterization of the constants in this relation. 

Inspecting the proofs of the lower bounds in Section \ref{sec:lower} reveals that, in a setting where no prior information on $\bt$ is given, the supervised learning estimator \eqref{oracle:hint} is optimal. Interestingly, supervised and unsupervised risks are almost equal, in the sense that the problem of estimating the community of a new observation $Y_{n+1}$ in the Gaussian mixture model is as hard with no information on the first $n$ labels as with
full supervised information on the first $n$ labels,
as long as the centers are unknown. 

As for the upper bound, we introduce and analyze a fully adaptive rate optimal and computationally simple procedure inspired from Lloyd's iterations. In order to achieve optimal decay of the risk, it turns out that it is enough to consider the hollowed Gram matrix
$\mathbf{H}(Y^{\top}Y)$
 where for any squared matrix $M$,  $\mathbf{H}(M)=M-\text{diag}(M)$ and $\text{diag}(M)$ is the diagonal of $M$. 
 
 Our approach has two building blocks. We start by finding a good initialization, then use few iterations aiming to improve the risk of our estimator. We set the initializer $\eta^{0}$ such that $\eta^{0}=\text{sign}(\hat{v})$ and $\hat{v}$ is the eigenvector corresponding to the top eigenvalue of $\mathbf{H}(Y^{\top}Y)$. The risk of $\eta^{0}$ is studied in Section \ref{sec:spectral}. In particular, we observe that $\eta^{0}$ can achieve almost full recovery but cannot show it is rate optimal. As an improvement, we consider in Section \ref{sec:upper} the iterative sequence of estimators $(\eta^{k})_{k \geq 1}$ defined as
 $$
 \forall k \geq 0,\quad \eta^{k+1} = \text{sign}\left( \mathbf{H}(Y^{\top}Y)\eta^{k}\right),
 $$
 and show that it achieves optimal exact recovery after a logarithmic number of iterations.

 In comparison to \cite{zhou}, we get better results, in particular for large $p$. In their approach, a spectral initialization on $\bt$ is considered and estimation of $\bt$ is handled at each iteration. The main difference compared to our procedure lies in the fact that we bypass the step of estimating $\bt$. We only use the matrix $\mathbf{H}(Y^\top Y)$ that is almost blind to the direction of $\bt$. \citet{giraud} present a rate optimal procedure without capturing the sharp optimality. Our procedure differs in two ways from \cite{giraud}. First, it is not an SDP relaxation method and hence is faster to compute. Second, by using the operator $\mathbf{H}$, we do not need to de-bias the Gram matrix, as diagonal deletion is sufficient in our case. In the case of isotropic noise, we should emphasize  that, while there is no need to de-bias the spectral part, as pointed out in \cite{giraud}, it seems to be necessary for Lloyd's iterations to achieve a rate optimal decay. In the general case of heteroscedastic  noise, it is not clear if diagonal deletion is enough to de-bias the Gram matrix.

 After our work was made publicly available, several follow-up papers have extended and complemented our results. On the one hand, \cite{chen2020cutoff} generalized the sharp phase transition to the case of $K$ communities where $K = \mathcal{O}(\log{n})$. On the other hand, after we characterized the sharp phase transition for exact recovery, \cite{leoffler} and \cite{abbe2020ell_p} showed that spectral clustering is sufficient to achieve optimal exact recovery. While their sharp results rely heavily on the Gaussian assumption of the noise, our analysis still holds in the case of sub-Gaussian noise. We comment on their findings in Section \ref{sec:simu} where we show that our procedure outperforms the spectral method empirically.

\section{Non-asymptotic fundamental limits in the Gaussian mixture model}\label{sec:lower}
In this section, we derive a sharp optimal lower bound for the risk $\Psi_{\Delta}$.  As stated in the Introduction, a simple lower bound is given by \eqref{eq:lower_trivial}. The next proposition provides a sharper statement.
\begin{proposition}\label{prop:lower_trivial} 
For any $\Delta>0$, we have 
$$
\Psi_{\Delta} \geq 
c\Phi^{{\sf c}}(\Delta/\sigma),
$$
for some absolute constant $c>0$.
\end{proposition}
Following the same lines as in \cite{ndaoud}, we obtain two different lower bounds for the minimax risk. Proposition \ref{prop:lower_trivial} gives a bound responsible for the hardness of recovering communities due to the lack of information on the labels. It still benefits from the knowledge of $\bt$. In \cite{giraud}, it becomes clear that for large $p$, the hardness of the problem is due to the hardness of estimating $\bt$. Hence, in order to capture this phenomenon, one may try to hide the information about the direction of $\bt$ in order to make its estimation difficult.

More precisely, in order to bound the risk $\Psi_\Delta$ from below, we place a prior on both $\eta$ and $\bt$. Intuitively, we would choose a Gaussian prior for $\bt$ in order to make its estimation the hardest, but one should keep in mind that $\bt$ is constrained to the set $\Omega_{\Delta}$. To derive lower bounds on constrained sets, we act as in \cite{butucea2018}. Let $\pi=\pi_{\bt}\times \pi_{\eta}$ be a product probability measure on $\mathbf{R}^{p}\times\{-1,1\}^{n}$ (a prior on $(\bt,\eta)$). We denote by $\mathbb{E}_{\pi}$ the expectation with respect to~$\pi$.
\begin{theorem}\label{prop:lower_smart}
Let $\Delta>0$ and $\pi=\pi_{\bt}\times\pi_{\eta}$ a product probability measure on $\mathbf{R}^{p}\times\{-1,1\}^{n}$. Then,
$$ \Psi_{\Delta} \geq c\left(\frac{1}{\lfloor n/2 \rfloor}\sum_{i=1}^{\lfloor n/2 \rfloor}\underset{\hat{T_{i}}\in [-1,1]}{\inf}  \mathbb{E}_{\pi}\mathbf{E}_{(\bt,\eta)}|\hat{T}_{i} - \eta_{i}| - \pi_{\bt}\left( \|\bt \| < \Delta \right)\right), $$
where $\inf_{\hat{T}_{i}\in [-1,1]}$ is the infimum over all  estimators $\hat{T}_{i}(Y)$ with values in $[-1,1]$ and $c>0$.
\end{theorem}
Theorem \ref{prop:lower_smart} is useful to derive non-asymptotic lower bounds for constrained minimax risks. For the corresponding lower bound to be optimal, we need the remainder term $\pi_{\bt}\left( \|\bt \| < \Delta \right)$ to be negligible. In other words, the prior on $\bt$ must ensure that $\|\bt\|$ is greater than $\Delta$ with high probability. This would make the problem of recovery easier. Hence, it is clear that there exists some trade-off concerning the choice of $\pi_{\bt}$.

Let $\pi^\alpha = \pi^{\alpha}_{\bt} \times \pi_{\eta}$ be a product prior on $\mathbf{R}^{p}\times\{-1,1\}^{n}$, such that $\pi^{\alpha}_{\bt}$ is the distribution of the Gaussian random vector with i.i.d. centered entries of variance $\alpha^{2}$, $\pi_{\eta}$ is the distribution of the vector with i.i.d. Rademacher entries, and $\bt$ is independent of $\eta$. For this specific choice of prior we get the following result.
\begin{proposition}\label{thm:supervised_lower_bound}
For any $\alpha>0$, we have for all $i=1,\dots,n$,
$$ \underset{\hat{T}_{i}\in [-1,1]}{\inf}  \frac{1}{n}\mathbb{E}_{\pi^{\alpha}}\mathbf{E}_{(\bt,\eta)}|\hat{T}_{i} - \eta_{i}| \geq  \frac{1}{n}\mathbb{E}_{\pi^{\alpha}}\mathbf{E}_{(\bt,\eta)}|\eta^{**}_{i} - \eta_{i}|, $$
where $\eta^{**}$ is a supervised learning oracle given by
$$ \forall i = 1,\dots,n, \quad  \eta^{**}_{i} = {\normalfont \text{sign}} \left(Y_{i}^{\top}\left(\sum_{j \neq i}\eta_{j}Y_{j}\right) \right). $$
\end{proposition}
It is interesting to notice that each entry of the supervised learning oracle $\eta^{**}$ only depends on $\bt$ through its best estimator under the Gaussian prior when the labels for other entries are known. The lower bound of Proposition \ref{thm:supervised_lower_bound} confirms the intuition that the supervised learning oracle is optimal in a minimax sense. For $\sigma>0$, define $G_{\sigma}$ by the relation:
\begin{equation}\label{def:G}
    \forall t \in \mathbf{R}, \quad G_{\sigma}(t,\bt) =\mathbf{P}\left( \left( \bt + \sigma\xi_{1}\right)^{\top}\left( \bt + \frac{\sigma}{n-1}\sum_{j=2}^{n}\xi_{j} \right) \leq \| \bt \|^{2}t     \right),
\end{equation}
where $\xi_{1},\dots,\xi_{n}$ are i.i.d.  standard Gaussian random vectors. Combining Theorem \ref{prop:lower_smart} and Proposition \ref{thm:supervised_lower_bound} and using the fact that all entries of the prior $\pi^{\alpha}$ are i.i.d. we obtain the next proposition.
\begin{proposition}\label{prop:lower:simple}
Let $\Delta>0$ and let $G_{\sigma}$ be the function defined in \eqref{def:G}. For any $\alpha > 0$, we have
$$
  \Psi_{\Delta} \geq c\mathbb{E}_{\pi^{\alpha}_{\bt}}G_{\sigma}(0,\bt) - c\mathbf{P}\left( \sum_{j=1}^{p} \varepsilon_{j}^{2}\leq \frac{\Delta^{2}}{\alpha^{2}} \right) ,
  $$
  where $\varepsilon_{j}$ are i.i.d. standard Gaussian random variables and $c>0$
  is an absolute constant.
\end{proposition}
 We are now ready to state the main result of this section. As explained in \cite{giraud}, the main limitation of the analysis in \cite{zhou} is partially due to the choice of the signal-to-noise ratio (SNR) as $\Delta/\sigma$. We use here the SNR $\mathbf{r}_{n}$ given in \eqref{def:SNR}. It is of the same order as the SNR presented in \cite{giraud}.
\begin{theorem}\label{thm:lower_asymptotic}
Let $\Delta>0$. For $n$ large enough, there exists a sequence $\epsilon_{n}$  such that $\epsilon_{n} = o(1)$ and  
$$ \Psi_{\Delta} \geq c\Phi^{{\sf c}}(\left( \mathbf{r}_{n}(1+\epsilon_{n})\right), $$
for some absolute constant $c>0$.
\end{theorem}
It is worth saying that the result of Theorem \ref{thm:lower_asymptotic} holds without any assumption on $p$ and can be interpreted in a non-asymptotic sense by replacing $\epsilon_{n}$ by some small $c>0$. Moreover, since $\mathbf{r}_n < \Delta/\sigma$, it improves upon the lower bound in Proposition \ref{prop:lower_trivial}. This improvement is most dramatic in the regime $\Delta^{2}/\sigma^{2}= o\left( p/n\right)$ that we call the hard estimation regime.
\section{Spectral initialization}\label{sec:spectral}
In this section, we analyze the non-asymptotic minimax risk of the spectral initializer $\eta^{0}$. As it is the case for the SDP relaxations of the problem, the matrix of interest is the Gram matrix $Y^{\top}Y$. It is well known that is suffers from a bias that grows with $p$. In \cite{royer}, a de-biasing procedure is proposed using an estimator of the covariance of the noise. This step is important to obtain a procedure adaptive to the noise level. Our approach is different but is still adaptive and consists in removing the diagonal entries of the Gram matrix. We give here some intuition about this procedure. Define the linear operator $\mathbf{H}:\mathbf{R}^{n \times n } \to \mathbf{R}^{n \times n }$ as follows:
$$
\forall M \in \mathbf{R}^{n \times n}, \quad \mathbf{H}(M) = M - \text{diag}(M),
$$
where $\text{diag}(M)$ is a diagonal matrix with the same diagonal as $M$. Going back to Proposition \ref{thm:supervised_lower_bound}, we may observe that the oracle $\eta^{**}$ can be written as  
\begin{equation}\label{oracle:hint}
    \eta^{**} = \text{sign}\left(\mathbf{H}\left(Y^{\top}Y\right)\eta\right),
\end{equation}
where the sign is applied entry-wise. This suggests that the matrix $\mathbf{H}\left(Y^{\top}Y\right)$ appears in a natural way. We can decompose $\mathbf{H}(Y^\top Y)$ as follows:
\begin{equation}\label{eq:decomp}
    \mathbf{H}(Y^{\top}Y) = \|\bt\|^{2} \eta\eta^{\top}  +  \mathbf{H}(W^{\top}W) + \mathbf{H}(W^{\top}\bt\eta^{\top} + \eta \bt^{\top}W ) - \|\bt\|^{2}\mathbf{I}_{n}.
\end{equation}
Apart from the scalar factor $\|\bt\|^{2}$, this expression is similar to the SBM or the symmetric spiked model, with the noise term $\mathbf{H}(Y^\top Y)$ having a more complex structure. It turns out that the main driver of the noise is $\mathbf{H}(W^{\top}W)$. A simple lemma (cf. Appendix \ref{sec:lem}) shows that our approach is an alternative to de-biasing the Gram matrix. Specifically, Lemma \ref{lem:diagonal} gives that
$$
\| \mathbf{H}(W^{\top}W) \|_{op} \leq 2\left\| W^{\top}W - \mathbf{E}\left( W^{\top}W\right) \right\|_{op}
$$
almost surely
for any random matrix $W$ with independent columns. Hence, the noise term can be controlled as if its covariance were known. Nevertheless, the operator $\mathbf{H}(.)$ may affect dramatically the signal since it also removes its diagonal entries. Fortunately, the signal term is almost insensitive to this operation since it is a rank-one matrix where the spike energy is spread all over all spike entries. For instance, we have
$$\|\mathbf{H}(\eta \eta^{\top})\|_{op} = \left(1-\frac{1}{n}\right)\|\eta \eta^{\top}\|_{op} .$$ Hence as $n$ grows the signal does not get affected by removing the diagonal terms while we get rid of the bias in the noise. It is worth noticing that our approach succeeds thanks to the specific form of $\eta$ and cannot be generalized to any spiked model. For the general case, a more consistent approach is proposed in \cite{wu_sparse}, where the diagonal entries can be used to achieve optimal estimation accuracy.
\noindent Motivated by \eqref{eq:decomp}, the spectral estimator $\eta^{0}$ is defined by
\begin{equation}\label{def:eta:zero}
    \eta^{0} = \text{sign}(\hat{v}),
\end{equation}
where $\hat{v}$ is the eigenvector corresponding to the top eigenvalue of $\mathbf{H}(Y^{\top}Y)$.
The next result characterizes the non-asymptotic minimax risk of $\eta^{0}$. 
\begin{theorem}\label{thm:spectral}
Let $\Delta>0$ and let $\eta^{0}$ be the estimator given by \eqref{def:eta:zero}. Under the condition $\mathbf{r}_{n} \geq C$, for some absolute constant $C>0$, we have
\begin{equation}\label{eq:sub_n_r}
\underset{(\bt,\eta)\in \Omega_{\Delta}}{\sup}\frac{1}{n}\mathbf{E}_{(\bt,\eta)}r(\eta^{0},\eta) \leq \frac{C'}{\mathbf{r}_{n}^{2}} + \frac{32}{n^{2}},
\end{equation}
and 
\begin{equation}\label{eq:decay}
\underset{(\bt,\eta)\in \Omega_{\Delta}}{\sup}\mathbf{P}_{(\bt,\eta)}\left( \frac{1}{n}\left|\eta^{\top}\eta^{0}\right| \leq 1 - \frac{\log{n}}{n} - \frac{C'}{\mathbf{r}_{n}^{2}} \right) \leq \epsilon_{n}\Phi^{{\sf c}}(\mathbf{r}_{n}),
\end{equation}
for some sequence $\epsilon_{n}$ such that $\epsilon_{n}=o(1)$ and some absolute constant $C'>0$.
\end{theorem}
As we may expect the appropriate Hamming distance risk is decreasing with respect to $\mathbf{r}_{n}$.  The condition $\mathbf{r}_{n} = \Omega(1)$ is very natural, since it is necessary even for detection as shown in \cite{banks}. The residual term $\frac{32}{n^{2}}$ is due to removing the diagonal and can be seen as the price to pay for adaptation.
When $\mathbf{r}_{n}$ gets larger than $n$, removing the diagonal terms may be sub-optimal as the corresponding error dominates in \eqref{eq:sub_n_r}. 

As $n,\mathbf{r}_{n} \to \infty$, $\eta^{0}$ achieves almost full recovery (cf. Definition \ref{def:possible}). We show later that this condition is optimal but cannot show that $\eta^{0}$ is rate optimal. In general spiked models, the rate decay \eqref{eq:decay} is optimal. We rely on asymptotic random matrix theory, to elaborate on this point. In \cite{benaych}, it is shown that, in the asymptotics when $p/n \to c \in (0,1]$ and when the noise is Gaussian, detection is possible only for $\Delta^{2}\geq \sqrt{c}\sigma^{2}$. Moreover, the asymptotic correlation between $\eta$ and its spectral approximation is given by $ \sqrt{1 - \frac{c\sigma^{2}+\Delta^{2}}{\Delta^{2}(1+\Delta^{2}/\sigma^{2})}}$. When $\mathbf{r}_{n}= \Omega(1)$, we observe that $ \frac{c\sigma^{2}+\Delta^{2}}{\Delta^{2}(1+\Delta^{2}/\sigma^{2})} \asymp \frac{1}{\mathbf{r}_{n}^{2}} $. Hence, the decay in Theorem \ref{thm:spectral} is expected for general spiked models, but not necessarily rate optimal in our specific setting.
 Strikingly, thanks to the specific structure of $\eta$, it is possible to make the previous decay exponentially small as shown in \cite{leoffler} and \cite{abbe2020ell_p}.

\section{A rate optimal practical algorithm}\label{sec:upper}
In this section, we present an algorithm that is minimax optimal, adaptive to $\Delta$ and $\sigma$ and faster than SDP relaxation. In the same spirit as in \cite{zhou}, we are tempted by using Lloyd's iterations. If properly initialized, Lloyd's algorithm may achieve the optimal rate under mild conditions after only a logarithmic number of steps. We present here a variant of Lloyd's iterations. Motivated by \eqref{oracle:hint}, and given an estimator $\hat{\eta}^{0}$, we define a sequence of estimators $(\hat{\eta}^{k})_{k \geq 0}$ such that
\begin{equation}\label{def:iter}
    \forall k \geq 0, \quad \hat{\eta}^{k+1} = \text{sign}\left( \mathbf{H}\left(Y^{\top}Y \right)\hat{\eta}^{k}\right).
\end{equation}
The classical Lloyd's iterations correspond to the procedure \eqref{def:iter}, where $ \mathbf{H}\left(Y^{\top}Y \right)$ is replaced by $Y^{\top}Y$. If the initialization is good in a sense that we describe below, then at each iteration $\hat{\eta}^{k}$ gets closer to $\eta$ and achieves the minimax optimal rate after a logarithmic number of steps. The logarithmic number of steps is crucial computationally in order to obtain the desired exponential upper bound, as it is the case in many other iterative procedures.  

\begin{theorem}\label{thm:initial_lloyds}
Let $\Delta>0$ and let $\hat{\eta}^{0}$ be an estimator satisfying
$$ \frac{1}{n} \eta^{\top}\hat{\eta}^{0} \geq 1 - \frac{C'}{\mathbf{r}_{n}^{2}}-\nu_{n} $$
for some $C'>0$ and $\nu_{n} = o(1)$. Let $(\hat{\eta}^{k})_{k\geq0}$ be the corresponding iterative sequence \eqref{def:iter}. If $\mathbf{r}_{n} \geq C$ for some $C>0$, then after $k=\lfloor 3\log{n}\rfloor$ steps, we have
$$
\underset{(\bt,\eta) \in \Omega_{\Delta}}{\sup}\mathbf{E}_{(\bt,\eta)}r(\hat{\eta}^{k},\eta) \leq C'\mathbf{r}_{n}^{2}\underset{\|\bt\|\geq\Delta}{\sup}G_{\sigma}\left( \epsilon_{n}+ \frac{C'}{\mathbf{r}_{n}},\bt\right) + \epsilon_{n}\Phi^{{\sf c}}(\mathbf{r}_{n}),
$$
for some sequence $\epsilon_{n}$ such that $\epsilon_{n} = o(1)$ and $C'>0$.
\end{theorem}
Recall that $G(t,\bt)$ is close to $G(0,\bt)$ for small $t$.
Theorem \ref{thm:initial_lloyds} can be interpreted as follows. Given a good initialization, the iterative procedure \eqref{def:iter} achieves an error close to the supervised learning risk within a logarithmic number of steps. Observing that under the condition $\mathbf{r}_{n} \geq C$ for some $C>0$, the spectral estimator $\eta^{0}$ is a good initializer, we state a general result showing that our variant of Lloyd's iterations initialized with a spectral estimator is minimax optimal. 
\begin{theorem}\label{thm:spectral_lloyds}
Let $\Delta>0$. Let $\eta^{0}$ be the spectral estimator defined in \eqref{def:eta:zero} and let $(\eta^{k})_{k \geq0}$ be the iterative sequence \eqref{def:iter}. Assume that $\mathbf{r}_{n} > C$ for some $C>0$. 
Then, after $k=\lfloor 3\log{n}\rfloor$ steps we have
$$
\underset{(\bt,\eta)\in \Omega_{\Delta}}{\sup}\mathbf{E}_{(\bt,\eta)}r(\eta^{k},\eta) \leq C'\Phi^{{\sf c}}\left(\mathbf{r}_{n}\left(1-\epsilon_{n}-\frac{C'\log{\mathbf{r}_{n}}}{\mathbf{r}_{n}}\right)\right),
$$
for some sequence $\epsilon_{n}$ such that $\epsilon_{n} = o(1)$ and $C'>0$.
\end{theorem}
Notice that the upper bound in Theorem \ref{thm:spectral_lloyds} is almost optimal, and gets closer to the optimal minimax rate as $n,\mathbf{r}_{n} \to \infty$. Hence, under mild conditions, we get a matching upper bound to the lower bound in Theorem \ref{thm:lower_asymptotic}. Moreover, we figure out that a good initialization combined with smart iterations is almost equivalent to the supervised learning oracle. In fact, the rate in Theorem \ref{thm:spectral_lloyds} is almost the same as the rate of the supervised oracle $\eta^{**}$. We conclude that unsupervised learning is asymptotically as easy as supervised learning in the Gaussian mixture model. The next proposition gives a full picture of the minimax risk $\Psi_{\Delta}$.
\begin{proposition}\label{prop:match_lower}
Let $\Delta>0$. For some absolute constants $c_{1},c_{2},C_{1},C_{2}>0$ and $n$ large enough, we have
$$
 C_{1}e^{-c_{1}\mathbf{r}^{2}_{n}} \leq \Psi_{\Delta} \leq C_{2}e^{-c_{2}\mathbf{r}^{2}_{n}}.
$$
\end{proposition}
Notice that the procedure we present here has a different rate of decay compared to the spectral procedure \eqref{def:eta:zero}. Our proof technique, makes it possible to turn an estimator with weak recovery guarantees into a sharply optimal one. Analysis of the iterations is almost deterministic as long as the noise is  sub-Gaussian. 

Recent papers by \cite{xia} and \cite{abbeentriwise} show that spectral algorithms can achieve exact recovery using refined sup-norm perturbation techniques. Although their results are surprising, they match optimal conditions for exact recovery in the Gaussian mixture model only in the zone $\mathbf{r}_{n} \asymp \Delta/\sigma$. 

\section{Asymptotic analysis. Phase transitions}\label{sec:sharp}

This section deals with asymptotic analysis of the problem of
community detection in the two component Gaussian mixture model. The results are derived as corollaries of the
minimax bounds of previous sections. We will assume that $n\to
\infty$ and that parameters $p, \sigma$ and $\Delta$ depend on $n$. For the sake of readability we do not equip some parameters with the index $n$.

The two asymptotic properties we study here are \textit{exact
recovery} and \textit{almost full recovery}. We establish the complete characterization of the sharp phase transition for both almost full and exact recovery. We use the terminology
following \cite{butucea2018} that we recall here.
\begin{definition}\label{def:possible} Let $(\Omega_{\Delta_{n}})_{n\geq 2}$ be a sequence of
classes corresponding to $(\Delta_{n})_{n \geq 2}$:
\begin{itemize}
\item We say that \emph{almost full recovery is possible} for
$(\Omega_{\Delta_{n}})_{n\geq 2}$ if there exists an estimator $\hat\eta$
such that
%
\begin{gather}
\label{af1}  \lim_{n\to\infty} \sup_{(\bt,\eta)\in \Omega_{\Delta_{n}}} \frac{1}{n}\mathbf{E}_{(\bt,\eta)}r(\hat{\eta},\eta)  =0.
\end{gather}
In this case, we say that $\hat\eta$ achieves almost full recovery.
\item We say that \emph{almost full recovery is impossible} for
$(\Omega_{\Delta_{n}})_{n\geq 2}$ if
%
\begin{gather}
\label{af2} \liminf_{n \to\infty} \inf_{\tilde{\eta}} \sup_{(\bt,\eta)\in \Omega_{\Delta_{n}}} \frac{1}{n}\mathbf{E}_{(\bt,\eta)}r(\tilde{\eta},\eta) >0,
\end{gather}
where $ \inf_{\tilde{\eta}}$ denotes the infimum over all estimators in $\{-1,1\}^{n}$.
\end{itemize}
\end{definition}

\begin{definition}\label{def:possibleexact} Let $(\Omega_{\Delta_{n}})_{n\geq 2}$ be a sequence of
classes corresponding to $(\Delta_{n})_{n \geq 2}$:
\begin{itemize}
\item We say that
\emph{exact recovery is possible} for $(\Omega_{\Delta_{n}})_{n\geq 2}$ if there exists an estimator $\hat\eta$ such that
%
\begin{gather}
\label{er1} \lim_{n\to\infty} \sup_{(\bt,\eta)\in \Omega_{\Delta_{n}}} \mathbf{E}_{(\bt,\eta)}r(\hat{\eta},\eta) =0.
\end{gather}
In this case, we say that $\hat\eta$ achieves exact recovery.

\item We say that
\emph{exact recovery is impossible} for $(\Omega_{\Delta_{n}})_{n\geq 2}$ if
%
\begin{gather}
\label{er2} \liminf_{n \to\infty} \inf_{\tilde{\eta}} \sup_{(\bt,\eta)\in \Omega_{\Delta_{n}}} \mathbf{E}_{(\bt,\eta)}r(\tilde{\eta},\eta)  >0,
\end{gather}
where $ \inf_{\tilde{\eta}}$ denotes the infimum over all estimators in $\{-1,1\}^{n}$.
\end{itemize}
\end{definition}
\noindent The following general characterization theorem is a straightforward
corollary of the results of previous sections.

\begin{theorem}\label{thm:asymp1}
\textup{(i)}
Almost full recovery is possible for $(\Omega_{\Delta_{n}})_{n\geq 2}$ if
and only if %
\begin{equation}
\label{eq1:t3a} \Phi^{{\sf c}}(\mathbf{r}_{n})\to0 \qquad\text{as } n
\to\infty.
\end{equation}
In this case, the estimator $\eta^{k}$ defined in \eqref{def:eta:zero}-\eqref{def:iter}, with $k=\lfloor3\log{n}\rfloor$, achieves almost full recovery.

\textup{(ii)}
Exact recovery is impossible for $(\Omega_{\Delta_{n}})_{n\geq 2}$ if for some $\epsilon>0$
%
\begin{equation}
\label{eq2:t3a} 
 \liminf_{n \to\infty}n\Phi^{{\sf c}}(\mathbf{r}_{n}(1+\epsilon))>0
\qquad\text{as } n\to\infty,
\end{equation}
and possible if for some $\epsilon>0$
\begin{equation}
\label{eq3:t3a} n\Phi^{{\sf c}}(\mathbf{r}_{n}(1-\epsilon))\to0
\qquad\text{as } n\to\infty,
\end{equation}
In this case, the estimator $\eta^{k}$ defined in \eqref{def:eta:zero}-\eqref{def:iter}, with $k=\lfloor3\log{n}\rfloor$, achieves exact recovery.
\end{theorem}
Although Theorem~\ref{thm:asymp1} gives a complete solution to both problems of almost full and exact recovery,
conditions \eqref{eq1:t3a}, \eqref{eq2:t3a} and \eqref{eq3:t3a} are not quite
explicit.
The next Theorem is a consequence of Theorem~\ref{thm:asymp1}. It
describes a ``phase transition'' for $\Delta_{n}$ in the problem of almost full
recovery.

\begin{theorem}\label{thm:asymp}
\begin{itemize}
    \item[(i)] If $\sigma^{2}\left(1+\sqrt{p/n}\right)=o(\Delta_{n}^{2})$,
        then the estimator $\eta^{k}$ defined in \eqref{def:eta:zero}-\eqref{def:iter}, with $k=\lfloor3\log{n}\rfloor$, achieves almost full recovery.
    \item[(ii)] Moreover if $\Delta_{n}^{2}=\mathcal{O}\left( \sigma^{2}(1+\sqrt{p/n})\right)$,
        then almost full recovery is impossible.
\end{itemize}
\end{theorem}
\noindent Theorem~\ref{thm:asymp} shows that almost
full recovery occurs if and only if
%
\begin{equation}
\label{phase}  \sigma^{2}\left(1+\sqrt{p/n}\right)=o(\Delta_{n}^{2}).
\end{equation}
As for exact recovery, the present work characterizes a precise threshold
$\bar{\Delta}_n$ such that exact recovery is possible
for $\Delta_n$ greater than $\bar{\Delta}_n$ and is impossible for $\Delta_n$ smaller than
$\bar{\Delta}_n$. 
Define $\bar{\Delta}_{n}>0$ such that
\begin{equation}\label{def:threshold}
    \bar{\Delta}^{2}_{n} = \sigma^{2}  \left( 1 + \sqrt{1 + \frac{2p}{n\log{n}}}\right)\log{n}. 
\end{equation}
\begin{theorem}\label{thm:asympE}
\begin{itemize}
\item[(i)]
Let $\Delta_{n} \geq \bar{\Delta}_{n}(1+\epsilon)$
for some $\epsilon>0$.
 Then, the estimator $\eta^{k}$ defined in \eqref{def:eta:zero}-\eqref{def:iter}, with $k=\lfloor3\log{n}\rfloor$, achieves exact recovery.
\item[(ii)] If the complementary condition holds, i.e, $\Delta_{n} \leq \bar{\Delta}_{n}(1-\epsilon)$
for some $\epsilon>0$, then exact recovery is impossible.
\end{itemize}
\end{theorem}
Some remarks are in order here. First of all, Theorem~\ref{thm:asympE}
shows that
the ``phase transition'' for exact recovery occurs exactly at $\bar{\Delta}_{n}$ given by \eqref{def:threshold}.
It is remarkable that this sharp threshold for exact recovery is valid
for all values of $p$. In particular, it gives rise to a  critical dimension $p^{*} = n\log{n}$:
\begin{itemize}
    \item 
If $p=o(p^{*}) $, then $
\bar{\Delta}_{n}=(1+o(1))\sigma\sqrt{2\log{n}}$. In this case, the phase transition threshold for exact recovery is the same as if $\bt$ were known. Indeed, according to lower bound \eqref{eq:lower_trivial}, it is straightforward that condition $
\Delta \geq \sigma\sqrt{2\log{n}}$ is necessary for exact recovery.
\item
 On the other hand if $p^{*}=o(p)$, then $\bar{\Delta}_{n}=(1+o(1))\sigma\left(\frac{2p\log{n}}{n}\right)^{1/4}$. This new condition takes into account the hardness of estimation, and $p^{*}$ can be interpreted as a phase transition with respect to the hardness of estimation of $\bt$.
\end{itemize}

\section{Numerical simulation results}\label{sec:simu}

The goal of this section is twofold. First, we verify empirically
    the sharp transition threshold $\bar\Delta_n$.
Second, we compare the performance of estimator \eqref{def:eta:zero}-\eqref{def:iter} that we call \textit{spectral Lloyd's} against the spectral method \eqref{def:eta:zero} that we call \textit{spectral}. In what follows, we fix $n=500$ the number of labels. For the sake of readability of plots, we define the parameters $a$ and $b $ such that

\[
\Delta^2 = (1+\sqrt{a})\log{n} \quad \text{and} \quad p = bn\log{n}.
\]
According to \eqref{def:threshold}, and using the above parameterization, the sharp phase transition for exact recovery happens at
\[
a = 1 + 2b.
\]
\noindent Our simulation setup was defined as follows. We set $a$ on a uniform grid of $50$ points delimited by $1.1$ and $11$. Similarly, we set $b$ on a uniform grid of $50$ points delimited by $0.1$ and $5$. For each combination of values of $ a$ and $b$, simulation was repeated $300$ times and we return the indicator of success i.e. when the estimator recovers exactly the true labels vector $\eta$. For a better interpretation, we apply the function $x \to 10^{-3(1-x)}$ to the average probability of exact recovery in the plots (Figure \ref{fig:comparison}) for spectral Lloyd's (left) and spectral (right). 

\begin{figure}[ht]
\centering
\hspace{-.cm}
\begin{subfigure}{}
  \centering
  \includegraphics[width=.45\linewidth]{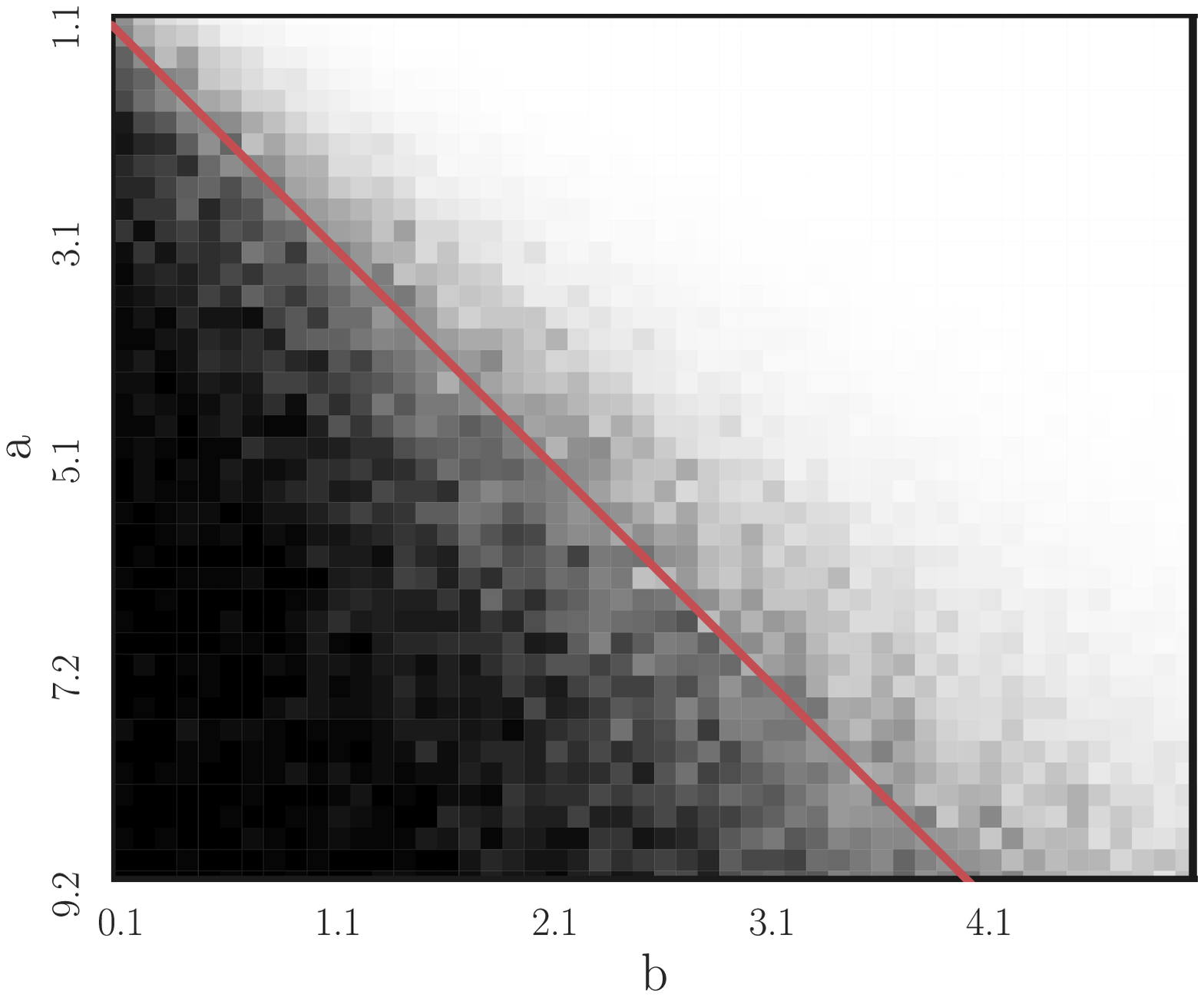}
  \label{fig:sub1}
\end{subfigure}
\centering
\begin{subfigure}
  \centering
  \includegraphics[width=.515\linewidth]{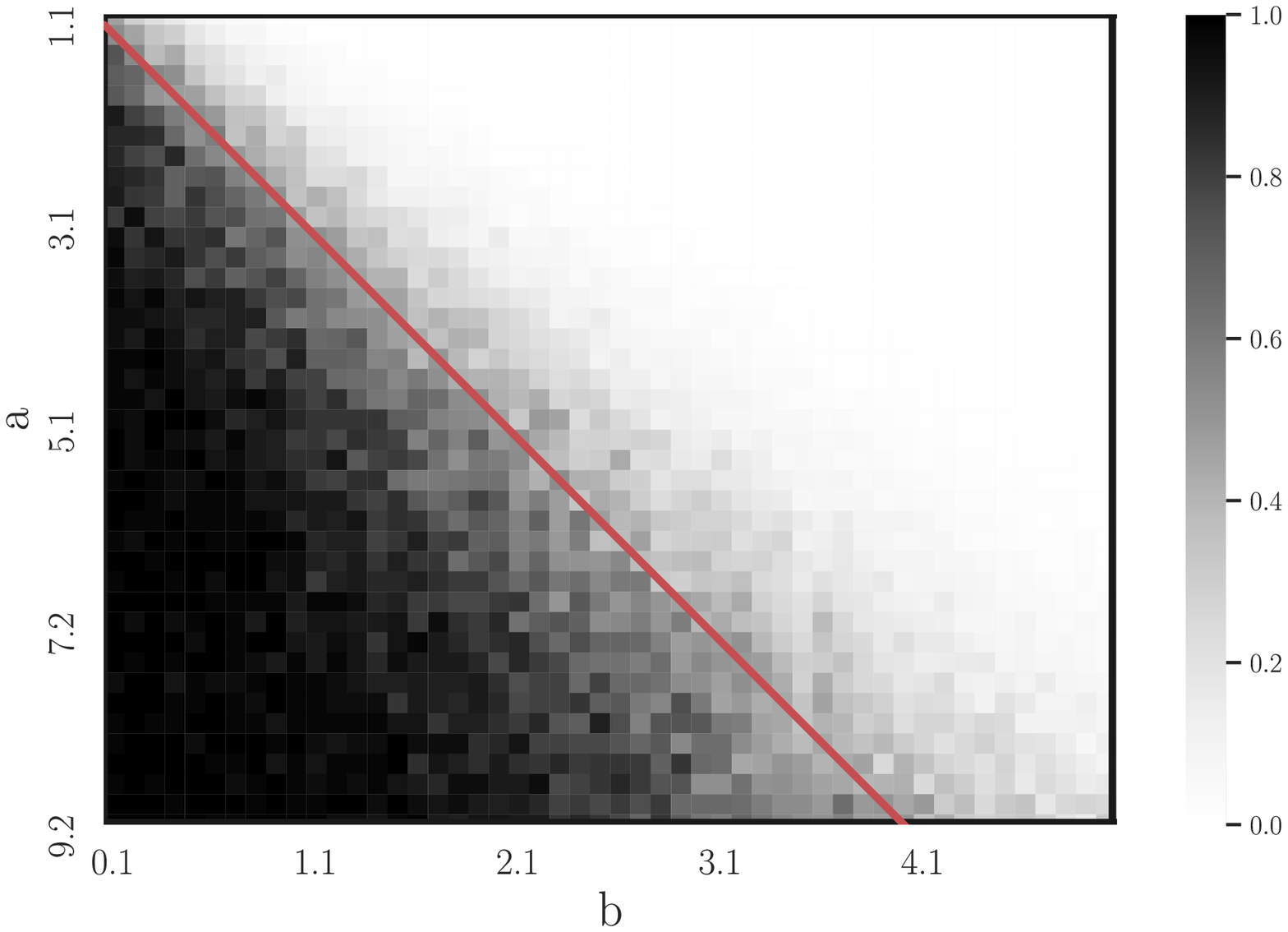}
  \label{fig:sub2}
\end{subfigure}

\caption{Empirical probability of success over 300 runs of the experiment for: spectral Lloyd's (left) and spectral (right). The red curve corresponds to the threshold equation $a = 1 + 2b$.}
\label{fig:comparison}
\end{figure}
As predicted by our theory, spectral Lloyd's achieves exact recovery with high probability at the sharp threshold given by \eqref{def:threshold}. Interestingly, the spectral method itself achieves exact recovery with high probability at the optimal threshold as claimed by \cite{leoffler} and \cite{abbe2020ell_p}.

In order to compare the two methods, we plot the difference of their respective average probability of exact recovery in the plots (Figure \ref{fig:comparison2}) in: the regime where exact recovery is impossible (left) and the regime where exact recovery is possible (right).
\begin{figure}[ht]
\centering
\hspace{-.cm}
\begin{subfigure}{}
  \centering
  \includegraphics[width=.465\linewidth]{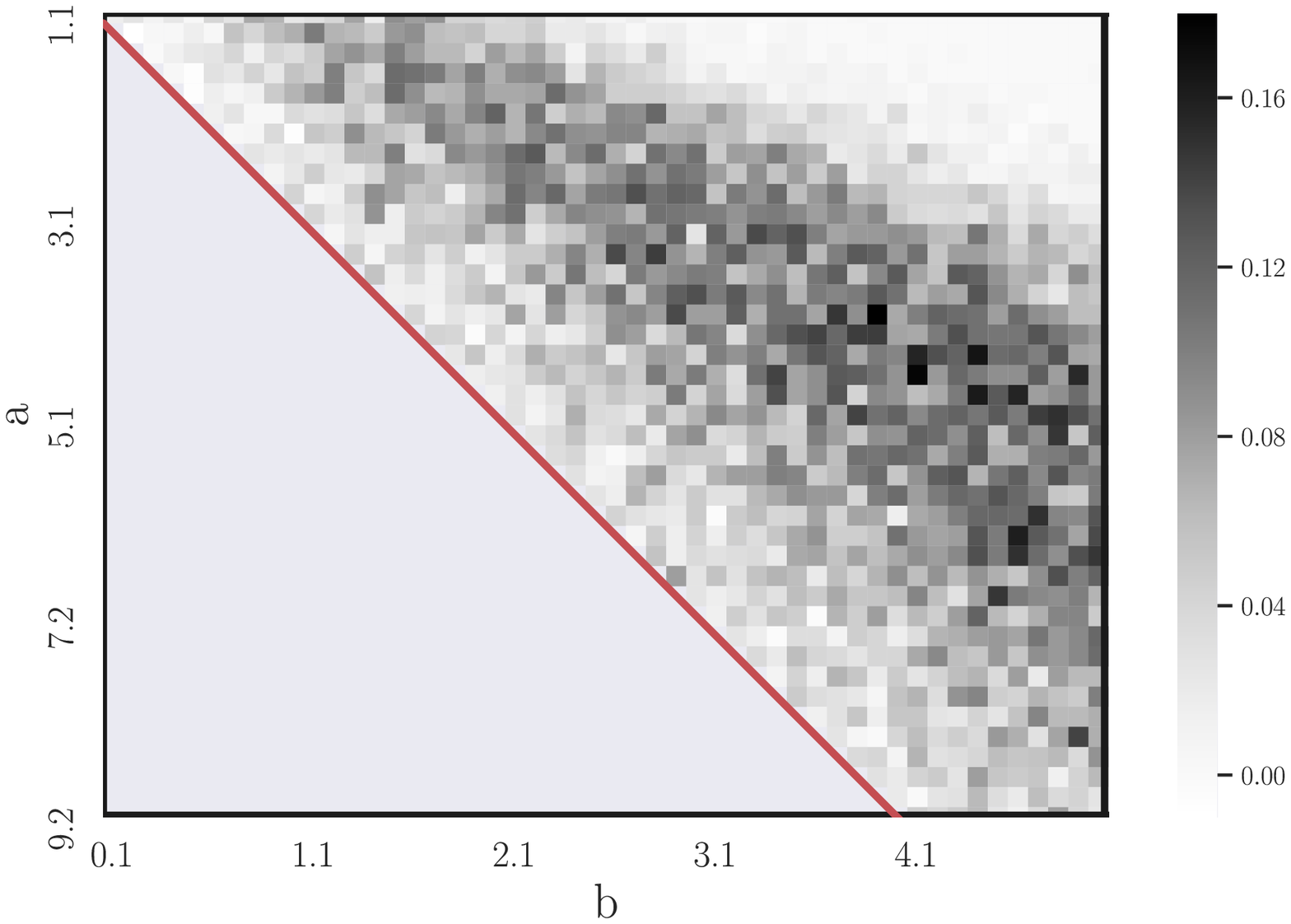}
  \label{fig:sub3}
\end{subfigure}
\centering
\begin{subfigure}{}
  \centering
  \includegraphics[width=.475\linewidth]{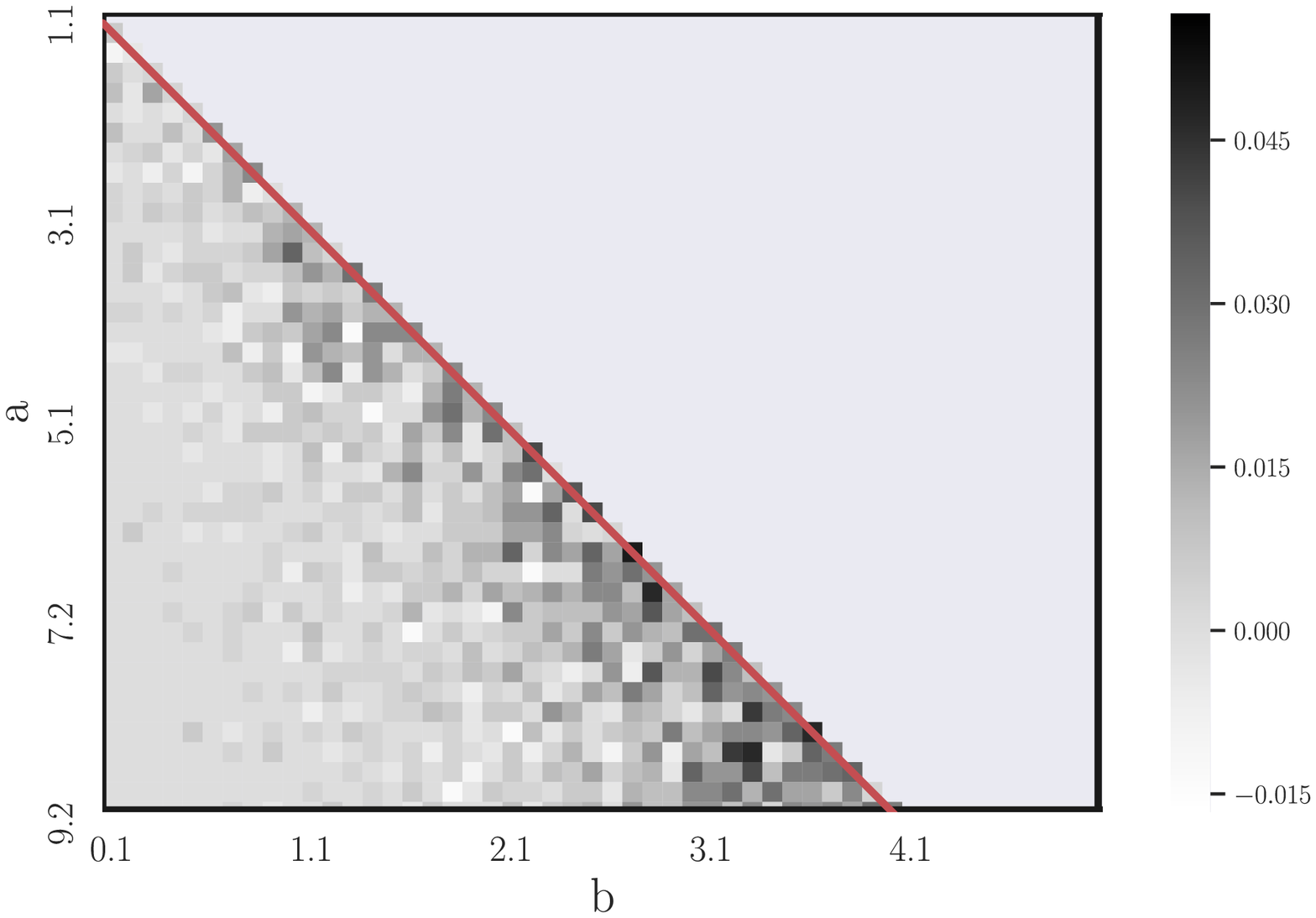}
  \label{fig:sub4}
\end{subfigure}

\caption{Difference between empirical probability of success of spectral Lloyd's and spectral over 300 runs of the experiment in: the regime of impossible exact recovery (left) and the regime where exact recovery is possible (right). The red curve corresponds to the threshold equation $a = 1 + 2b$.}
\label{fig:comparison2}
\end{figure}

In both regimes, spectral Lloyd's outperforms spectral especially as we get closer to the threshold equation $a=1+2b$. In the regime of impossible exact recovery, spectral Lloyd's captures more information about the labels. 

Overall, numerical experiments confirm our theoretical findings. Moreover, empirical comparison of spectral Lloyd's with spectral suggests that it is still useful to use Lloyd's iteration to improve upon the spectral initialization.

\section{Discussion and open problems}
A key objective of this paper was to characterize the sharp phase transition threshold for exact recovery in the two component Gaussian mixture model. All upper bounds remain valid in the case of sub-Gaussian noise. It would be interesting to generalize the methodology used to derive both lower and upper bounds to the case of multiple communities and general covariance structure of the noise. We also expect the procedure \eqref{def:eta:zero}-\eqref{def:iter} to achieve exact recovery in asymptotically sharp way in other problems, for instance in the Bipartite Stochastic Block Model as in \cite{ndaoud2019improved}. 

We conclude this paper with some open questions. Let $p^{*}=n\log{n}$. In the regime $p^{*}=o(p)$, we proved that for any $\epsilon>0$, the condition
$$\Delta^{2} \geq (1-\epsilon)\sigma^{2}\left(\frac{2p}{p^{*}}\right)^{1/2}\log{n}
$$
is necessary to achieve exact recovery. This is a consequence of considering a Gaussian prior on $\bt$ which makes recovering its direction the hardest. We give here a heuristics that this should hold independently on the choice of prior as long as $\bt$ is uniformly well-spread (i.e., not sparse). Suppose that we put a Rademacher prior on $\bt$ such that $\bt = \frac{\Delta}{\sqrt{p}}\zeta$, where $\zeta$ is a random vector with i.i.d. Rademacher entries. Following the same argument as in Proposition \ref{prop:lower_trivial}, it is clear that a necessary condition to get non-trivial correlation with $\zeta$ is given by $
\Delta^{2} \geq c\sigma^{2}\frac{p}{n}
$
for some $c>0$. Observing that, in the hard estimation regime, we have
$$
\left(\frac{p}{p^{*}}\right)^{1/2}\log{n} = o\left(\frac{p}{n}\right).
$$
It comes out that, while exact recovery of $\eta$ is possible, non-trivial correlation with $\zeta$ is impossible. Consequently, there is no hope achieving exact recovery through non-trivial correlation with $\bt$ in the hard estimation regime. It would be interesting to prove or disprove that for any $\epsilon>0$, 
$$ \Delta^{2} \geq (1-\epsilon) \sigma^{2}\sqrt{\frac{2p\log{n}}{n}} $$
is necessary to achieve exact recovery. In particular, a positive answer to the previous question will be very useful to derive optimal conditions for exact recovery in bipartite graph models, among other problems. 

Another interesting problem, is motivated by the following observation. After running simulations in Section \ref{sec:simu}, we decided to compare Lloyd's iterations initialized with a spectral method against randomly initialized Lloyd's iterations (\textit{random Lloyd's}). The plot in Figure \ref{fig:open} corresponds to the difference of the average probability of exact recovery between spectral Lloyd's and random Lloyd's.
\begin{figure}[ht]
\centering

  \includegraphics[width=.5\linewidth]{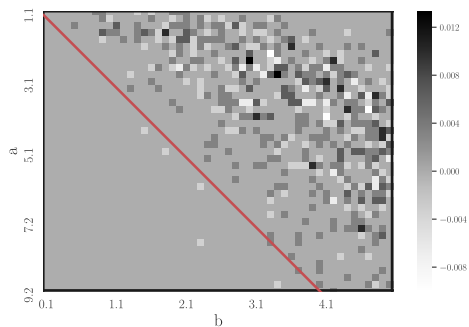}

\caption{Difference between empirical probability of success of spectral Lloyd's and random Lloyd's over 300 runs of the experiment. The red curve corresponds to the threshold equation $a=1+2b$.}
\label{fig:open}
\end{figure}
In particular, Figure \ref{fig:open} suggests that a good initialization is not required for Lloyd's iterations to achieve exact recovery.  It would be interesting to provide theoretical guarantees that the Lloyd's algorithm randomly initialized is already optimal, in the same spirit as in \cite{wu2019randomly}.

 \section*{Acknowledgements}

 I would like to thank Christophe Giraud for stimulating discussions on clustering in Gaussian mixtures and Alexandre Tsybakov  for valuable comments on early versions of this manuscript.  I also would like to thank Pierre Bellec for helpful discussions that improved the presentation of this paper. The first version of this work (Ndaoud 2018, arXiv:1812.08078v1) was supported by Investissements d'Avenir (ANR-11-IDEX-0003/Labex Ecodec/ANR-11-LABX-0047).

\bibliographystyle{imsart-nameyear}

\appendix

\section{Main proofs}

In all the proofs of lower bounds, we follow the same argument as in Theorem $2$ in \cite{gao2018} in order to substitute the minimax risk of $r(\tilde{\eta},\eta)$ by a Hamming minimax risk. Let $z^{*}$ be a vector of labels in $\{-1,1\}^{n}$ and let $T$ be a subset of $\{1,\dots,n\}$ of size $\lfloor n/2 \rfloor$+1. A lower bound of the minimax risk is given on the subset of labels $\mathbf{Z}$, such that for all $i \in T$,we have $\eta_{i}=z^{*}_{i}$.  Observe that  in that case 
$$r(\eta^{1},\eta^{2}) = |\eta^{1} - \eta^{2}|,$$
for any $\eta^{1},\eta^{2} \in \mathbf{Z}$. The argument in \cite{gao2018}, leads to
$$
\Psi_{\Delta} \geq \frac{c}{|T^{{\sf c}}|}\sum_{i \in T^{{\sf c}}}\underset{\tilde{\eta}_{i}}{\inf}\mathbb{E}_{\pi}\mathbf{E}_{(\bt,\eta)}|\tilde{\eta}_{i}-\eta_{i}|,
$$
for some $c>0$ and for any prior $\pi$ such that $\pi_{\eta}$ is invariant by a sign change. That is typically the case under Rademacher prior on labels. As a consequence, a lower bound of $\Psi_{\Delta}$ is given by a lower bound of the R.H.S minimax Hamming risk. 

\noindent Rescaling $\bar{\bt}$ by $\sigma$, we also assume that $\sigma = 1$, without loss of generality.
\subsection{Proof of Proposition \ref{prop:lower_trivial}}
Let $\bar{\bt}$ be a vector in $\mathbf{R}^{p}$ such that $\|\bar{\bt}\|=\Delta$. Placing an independent Rademacher prior $\pi$ on $\eta$, and fixing $\bt$, it follows that
\begin{equation}\label{eq:proof:triv:1}
\underset{\tilde{\eta}_{j}}{\inf}\mathbb{E}_{\pi}\mathbf{E}_{(\bar{\bt},\eta)}|\tilde{\eta}_{j}- \eta_{j}| \geq \underset{\bar{\eta}_{j}}{\inf}\mathbb{E}_{\pi}\mathbf{E}_{(\bar{\bt},\eta)}|\bar{\eta}_{j}(Y_{j}) - \eta_{j}|,
\end{equation}
where $\bar{\eta}_{j} \in [-1,1]$.
The last inequality holds because of independence between the priors. We define, for $\epsilon \in \{-1,1\}$, $\tilde{f}_{\epsilon}(.)$ the density of the observation $Y_{j}$ conditionally on the value of $
\eta_{j}=\epsilon$. Now, using Neyman-Pearson lemma and the explicit form of $\tilde{f}_{\epsilon}$, we get that the selector $\eta^{*}$ given by
$$
\eta_{j}^{*} = \text{sign}\left( \bar{\bt}^\top Y_{j}\right), \quad \forall j=1,\dots,n
$$
is the optimal selector that achieves the minimum of the RHS of \eqref{eq:proof:triv:1}.
Plugging this value in \eqref{eq:proof:triv:1}, we get further that 
$$
\underset{\bar{\eta}_{j}}{\inf}\mathbf{E}_{\pi}|\bar{\eta}_{j}(Y_{j}) - \eta_{j}| = 2\Phi^{{\sf c}}(\Delta).
$$
\subsection{Proof of Theorem \ref{prop:lower_smart}}
Throughout the proof, we write for brevity $A=\Omega_\Delta$.
Set $ \eta^{A} = \eta \mathbf{1}((\bt,\eta) \in A)$ and denote by $\bar{\pi}_A$ the probability measure $\pi$ conditioned by the event $\{(\bt,\eta) \in A\}$,  that is, for any $C\subseteq \mathbf{R}^{p}\times\{-1,1\}^{n}$,
$$
\bar{\pi}_A(C) = \frac{\pi (\{(\mathbf{\bt},\eta) \in C\}\cap \{(\bt,\eta) \in A\})}{\pi((\bt,\eta) \in A)}\,.
$$
The measure $\bar{\pi}_A$ is supported on $A$ and we have
\begin{eqnarray*}
\inf_{\tilde{\eta}_{j}} \mathbb{E}_{\bar{\pi}_{A}}\mathbf{E}_{(\bt,\eta)} |\tilde{\eta}_j - \eta_j |
&\ge&
\inf_{\tilde{\eta}_{j}} \mathbb{E}_{\bar{\pi}_{A}}\mathbf{E}_{(\bt,\eta)} |\tilde{\eta}_j - \eta^A_j |
\\
&\ge&   \inf_{\hat{T_j}} \mathbb{E}_{\bar{\pi}_{A}}\mathbf{E}_{(\bt,\eta)} |\hat{T}_j - \eta^A_j |
\end{eqnarray*}
where $\inf_{\hat T_j}$ is the infimum over all estimators $\hat T_j= \hat T_j(Y)$ with values in $\mathbf{R}$.
According to
Theorem~1.1 and Corollary~1.2 on page 228 in \cite{LC}, there exists a Bayes estimator $B^A_j=B^A_j(Y)$ such that
$$
\inf_{\hat{T_j}} \mathbb{E}_{\bar{\pi}_{A}}\mathbf{E}_{(\bt,\eta)} |\hat{T}_j - \eta^A_j |=\mathbb{E}_{\bar{\pi}_{A}}\mathbf{E}_{(\bt,\eta)} |B^A_j - \eta^A_j |,
$$
and this estimator is a conditional median of $\eta^A_j$ given $Y$.
Therefore,
\begin{equation}\label{eeq1}
\Psi_{\Delta}
\geq c\left(\frac{1}{\lfloor n/2 \rfloor}\sum_{j=1}^{\lfloor n/2 \rfloor}\mathbb{E}_{\bar{\pi}_{A}}\mathbf{E}_{(\bt,\eta)}  | B^A_j - \eta^{A}_{j}| \right).
\end{equation}
Note that $B^A_j\in [-1,1]$ since $\eta^A_j$ takes its values in $[-1,1]$.
Using this, we obtain
\begin{align}\nonumber
  \inf_{ \hat{T}_{j}\in[-1,1]} \mathbb{E}_{\pi}\mathbf{E}_{(\bt,\eta)} | \hat{T}_{j} - \eta_{j} | & \leq \mathbb{E}_{\pi}\mathbf{E}_{(\bt,\eta)} | B^A_j - \eta_{j}| \\ \nonumber
    & = \mathbb{E}_{\pi}\mathbf{E}_{(\bt,\eta)} \Big( |B^A_j - \eta_{j}| \mathbf{1}((\bt,\eta) \in A) \Big) + \mathbb{E}_{\pi}\mathbf{E}_{(\bt,\eta)}
   \Big(|B^A_j - \eta_{j}|  \mathbf{1}((\bt,\eta) \in A^c) \Big)
  \\  \nonumber
  & \leq \mathbb{E}_{\bar{\pi}_{A}}\mathbf{E}_{(\bt,\eta)}  |B^A_j - \eta_{j}^A|  + \mathbb{E}_{\pi}\mathbf{E}_{(\bt,\eta)}
   \Big( |B^A_j - \eta_{j}|  \mathbf{1}((\bt,\eta) \in A^c)  \Big)
   \\  \label{eeq2}
  & \le \mathbb{E}_{\bar{\pi}_{A}}\mathbf{E}_{(\bt,\eta)} |B^A_j - \eta_{j}^A|  + 2\mathbf{P}((\bt,\eta) \not\in A) .
 \end{align}
 The result follows combining \eqref{eeq1} and \eqref{eeq2}.

\subsection{Proof of Proposition \ref{thm:supervised_lower_bound}}
We start by using the fact that
$$ \mathbb{E}_{\pi^{\alpha}} \mathbf{E}_{(\bt,\eta)}|\hat{\eta}_{i} - \eta_{i}| =\mathbb{E}_{p_{-i}} \mathbb{E}_{p_{i}}\mathbf{E}_{(\bt,\eta)}\left( |\hat{\eta}_{i} - \eta_{i}| |(\eta_{j})_{j \neq i} \right), $$
where $p_{i}$ is the marginal of $\pi^{\alpha}$ on ($\bt$,$\eta_{i}$), while $p_{-i}$ is the marginal of $\pi^{\alpha}$ on $(\eta_{j})_{j \neq i} $. Using the independence between different priors, one may observe that $\pi^{\alpha} = p_{i}\times p_{-i}$. We define, for $\epsilon \in \{-1,1\}$, $\tilde{f}^{i}_{\epsilon}$ the density of the observation $Y$ given $(\eta_{j})_{j \neq i} $ and given $\eta_{i}=\epsilon$. Using Neyman-Pearson lemma, we get that
$$
\eta^{**}_{i} = \left\{
    \begin{array}{ll}
        1 & \mbox{if } \tilde{f}^{i}_{1}(Y) \geq \tilde{f}^{i}_{-1}(Y),\\
        -1 & \mbox{else,}
    \end{array}
\right.
$$
minimizes $\mathbb{E}_{p_{i}}\mathbf{E}_{(\bt,\eta)}\left( |\hat{\eta}_{i} - \eta_{i}| |(\eta_{j})_{j \neq i} \right)$ over all functions of $(\eta_{j})_{j\neq i}$ and of $Y$ with values in $[-1,1]$.
Using the independence of the rows of $Y$ we have
$$ \tilde{f}^{i}_{\epsilon}(Y) = \prod_{j=1}^{p}\frac{e^{-\frac{1}{2}L_{j}^{\top}\Sigma_{\epsilon}^{-1}L_{j}}}{(2\pi)^{p/2}|\Sigma_{\epsilon}|}, $$
where $L_{j}$ is the $j$-th row of $Y$ and $\Sigma_{\epsilon}= \mathbf{I}_{n} + \alpha^{2} \eta_{\epsilon} \eta_{\epsilon}^{\top}$. We denote by $\eta_{\epsilon}$ the binary vector such that $\eta_{i} = \epsilon$ and the other components are known. It is easy to check that $|\Sigma_{\epsilon}|=1+\alpha^{2}n$, hence it does not depend on $\epsilon$. A simple calculation leads to
$$ \Sigma_{\epsilon}^{-1} = \mathbf{I}_{n}- \frac{\alpha^{2}}{1+\alpha^{2}n}\eta_{\epsilon}\eta_{\epsilon}^{\top}. $$
Hence 

\begin{equation}
\begin{aligned}
\frac{\tilde{f}^{i}_{1}(Y)}{\tilde{f}^{i}_{-1}(Y)} &= \prod_{j=1}^{p}e^{-\frac{1}{2}L_{j}^{\top}(\Sigma_{1}^{-1} - \Sigma_{-1}^{-1})L_{j}} \nonumber \\
& = \prod_{j=1}^{p}e^{\frac{\alpha^{2}}{1+\alpha^{2}n}L_{ji}\sum_{k \neq i}L_{jk}\eta_{k}} \nonumber \\
&= e^{\frac{\alpha^{2}}{1+\alpha^{2}n}\sum_{k \neq i} \eta_{k} \sum_{j=1}^{p}L_{jk}L_{ji}} = e^{\frac{\alpha^{2}}{1+\alpha^{2}n} \langle Y_{i},\sum_{k \neq i}\eta_{k}Y_{k} \rangle } . \nonumber
\end{aligned}
\end{equation}
It is now immediate that
$$ \eta^{**}_{i} = \text{sign}\left(  Y_{i}^{\top} \left(\sum_{k \neq i}\eta_{k}Y_{k}\right)  \right). $$
\subsection{Proof of Proposition  \ref{prop:lower:simple}}
Combining Theorem \ref{prop:lower_smart} and Proposition \ref{thm:supervised_lower_bound}, we get that
$$ \Psi_{\Delta} \geq c\left( \frac{1}{\lfloor n/2 \rfloor}\sum_{i=1}^{\lfloor n/2 \rfloor}\mathbb{E}_{\pi^{\alpha}}\mathbf{E}_{(\bt,\eta)}|\eta^{**}_{i} - \eta_{i}| - \pi_{\bt}^{\alpha}\left( \|\bt \| \leq \Delta \right)\right).$$
Recall that here $\bt$ has i.i.d. centered Gaussian entries with variance $\alpha^{2}$. This yields the second term on the R.H.S of the inequality of Proposition \ref{prop:lower:simple}. While, for the first term, one may notice that the vectors $\eta_{i}Y_{i}$ for $i=1,\dots,n$ are i.i.d. and that
$$
|\eta^{**}_{i} - \eta_{i}| = 2\mathbf{1}\left(\eta_{i}Y_{i}^{\top}\left(\sum_{j\neq i}\eta_{j}Y_{j} \right) \leq 0 \right).
$$
Then, we use the definition of $G_{\sigma}$ \eqref{def:G} in order to conclude.
\subsection{Proof of Theorem \ref{thm:lower_asymptotic}}
 We prove the result by considering separately the following three cases.
\begin{enumerate}
    \item Case $\Delta \leq \frac{\log^{2}(n)}{\sqrt{n}}$.
In this case we use Proposition \ref{prop:lower_trivial}.

Since $0 \leq \frac{\Delta^{2}}{\sqrt{\Delta^{2}+p/n}} \leq \Delta$, we have $\left|\Delta - \frac{\Delta^{2}}{\sqrt{\Delta^{2}+p/n}}\right| \leq \frac{\log^{2}(n)}{\sqrt{n}}$. Hence
$$
\left| \Phi^{{\sf c}}(\Delta) - \Phi^{{\sf c}}\left( \frac{\Delta^{2}}{\sqrt{\Delta^{2}+p/n}}\right) \right| \leq c \frac{\log^{2}(n)}{\sqrt{n}}\Phi^{{\sf c}}\left( \frac{\Delta^{2}}{\sqrt{\Delta^{2}+p/n}}\right),
$$
for some $c>0$. Hence we get the result with $\epsilon_{n} = c \frac{\log^{2}(n)}{\sqrt{n}}$.
    \item Case $\Delta \geq \sqrt{\frac{p\log{n}}{n}}$.
In this case, we have $\sqrt{1+\frac{p}{n\Delta^{2}}}\frac{\Delta^{2}}{\sqrt{\Delta^{2}+p/n}} = \Delta$. It is easy to check that
$$
\left| \sqrt{1+\frac{p}{n\Delta^{2}}} - 1 \right| \leq \frac{1}{\log{n}}.
$$
Hence 
$$ \Delta \leq \frac{\Delta^{2}}{\sqrt{\Delta^{2}+p/n}}(1+\epsilon_{n}), $$
for $\epsilon_{n} = \frac{1}{\log{n}}$. We conclude using Proposition \ref{prop:lower_trivial}.
    \item Case $\frac{\log^{2}(n)}{\sqrt{n}} < \Delta < \sqrt{\frac{p\log{n}}{n}}$.
Notice that $p \geq \log^{3}(n)$ in this regime. We will use Proposition \ref{prop:lower:simple}. Set $\alpha^{2}$ such that
$$
\alpha^{2} = \frac{\Delta^{2}}{p(1-\nu_{n})} \quad \text{and} \quad \nu_{n} = \sqrt{\frac{n\Delta^{2}}{p\log^{2}(n)}}.
$$
It is easy to check that $0<\nu_{n}^{2} \leq 1/\log{n}$, Hence
$$ \mathbf{P}\left( \sum_{j=1}^{p} \varepsilon_{j}^{2}\leq \frac{\Delta^{2}}{\alpha^{2}} \right) = \mathbf{P}\left( \frac{1}{p}\sum_{j=1}^{p} (\varepsilon_{j}^{2}-1)\leq -\nu_{n}\right)\leq e^{-c\frac{n}{\log^{2}(n)}\Delta^{2}}, $$
for some $c>0$. 
Hence, for any $\epsilon_{n}\to 0$ we have
$$
\mathbf{P}\left( \sum_{j=1}^{p} \varepsilon_{j}^{2}\leq \frac{\Delta^{2}}{\alpha^{2}} \right) \leq e^{-c'\log{n}}\Phi^{{\sf c}}\left( \Delta(1+\epsilon_{n})\right) \leq e^{-c'\log{n}}\Phi^{{\sf c}}\left( \frac{\Delta^{2}}{\sqrt{\Delta^{2}+p/n}}(1+\epsilon_{n})\right) , 
$$
for some $c'>0$. Since $e^{-c'\log{n}} \underset{n \to \infty}{\to} 0$, then in order to conclude, we just need to prove that
$$ \mathbb{E}_{\pi_{\bt}^{\alpha}}G_{\sigma}(0,\bt) \geq (1-\epsilon_{n})\Phi^{{\sf c}}\left( \frac{\Delta^{2}}{\sqrt{\Delta^{2}+p/n}}(1+\epsilon_{n})\right), $$
for some sequence $\epsilon_{n}\to0$.

We recall that
$$ \mathbb{E}_{\pi_{\bt}^{\alpha}}G_{\sigma}(0,\bt) = \mathbf{P}\left( (\bt + \xi_{1})^{\top}\left( \bt + \frac{\xi_{2}}{\sqrt{n-1}}\right) \leq 0 \right),  $$
where $\xi_{1},\xi_{2}$ are two independent random vectors with i.i.d. standard Gaussian entries and $\bt$ is an independent Gaussian prior. Moreover, using independence, we have
\begin{align*} 
\mathbf{P}&\left( (\bt + \xi_{1})^{\top}\left( \bt + \frac{\xi_{2}}{\sqrt{n-1}}\right) \leq 0 \right)= \\ &\mathbf{P}\left( \varepsilon\sqrt{\|\bt\|^{2} + \frac{\|\xi_{2}\|^{2}}{n-1} + \frac{2}{\sqrt{n-1}} \bt^{\top}\xi_{2}} \geq  \|\bt\|^{2} + \frac{1}{\sqrt{n-1}} \bt^{\top}\xi_{2} \right),  
\end{align*}
where $\varepsilon$ is a standard Gaussian random variable. Fix $\bt$ and define the random event
$$
\mathcal{A} = \left\{ \frac{\|\xi_{2}\|^{2}}{n-1} \geq \frac{p}{n-1}(1-\zeta_{n}) \right\}\cap\left\{ |\bt^{\top}\xi_{2}| \leq \sqrt{n-1}\beta_{n}\|\bt\|^{2} \right\},
$$
where $\beta_{n}>0$ and $\zeta_{n} \in (0,1)$.
It is easy to check that 
\begin{equation}\label{eq:25a}
\mathbf{P}\left( \mathcal{A}^{{\sf c}}\right) \leq e^{-c\log^{3}(n)\zeta_{n}^{2}} + e^{-c\beta_{n}^{2}n\|\bt\|^{2}},
\end{equation}
for some $c>0$. Hence conditioning on $\bt$, we have
$$
\mathbf{P}\left( (\bt + \xi_{1})^{\top}\left( \bt + \frac{\xi_{2}}{\sqrt{n-1}}\right) \leq 0 \right) \geq \mathbf{E}\left[\Phi^{{\sf c}}\left(\frac{ \|\bt\|^{2}(1+\beta_{n}) }{\sqrt{\|\bt\|^{2}(1-2\beta_{n}) + \frac{p}{n-1}(1-\zeta_{n})  }}         \right) \mathbf{P}(\mathcal{A})\right].
$$
where the last expectation is over $\bt$. Define now the random event $\mathcal{B}=\left\{ \left| \|\bt\|^{2}   - \Delta^{2}\right| \leq \Delta^{2}\gamma_{n}\right\}$ where $\gamma_{n}\in(0,1)$. Then, using \eqref{eq:25a}, we get
\begin{equation}\label{eq:lower:hard}
\mathbf{P}\left( (\bt + \xi_{1})^{\top}\left( \bt + \frac{\xi_{2}}{\sqrt{n-1}}\right) \leq 0 \right) \geq \Phi^{{\sf c}}\left( U_{n} \right)\left(1- e^{-c\log^{3}{(n)}\zeta_{n}^{2}} - e^{-c\beta_{n}^{2}(1-\gamma_{n})\log^{4}(n)}\right) \mathbf{P}(\mathcal{B}),
\end{equation}
where $U_{n} := \frac{ \Delta^{2}(1+\beta_{n})(1+\gamma_{n}) }{\sqrt{\Delta^{2}(1-2\beta_{n})(1-\gamma_{n}) + \frac{p}{n-1}(1-\zeta_{n})  }} . $
Now we may check that
$$
\mathbf{P}(\mathcal{B}^{{\sf c}}) = \mathbf{P}\left( \left|\sum_{j=1}^{p} \varepsilon_{j}^{2} - \frac{\Delta^{2}}{\alpha^{2}} \right|\geq \frac{\Delta^{2}}{\alpha^{2}}\gamma_{n} \right).
$$
Hence
$$
\mathbf{P}(\mathcal{B}^{{\sf c}}) \leq \mathbf{P}\left( \left|\sum_{j=1}^{p} \varepsilon_{j}^{2} - p\right|\geq \frac{\Delta^{2}}{\alpha^{2}}\gamma_{n} - \left|p - \frac{\Delta^{2}}{\alpha^{2}}\right|  \right).
$$
Using the definition of $\alpha^{2}$ we get 
\begin{equation}\label{eq:lower:hard:2}
\mathbf{P}(\mathcal{B}^{{\sf c}}) \leq \mathbf{P}\left( \left|\sum_{j=1}^{p} \varepsilon_{j}^{2} - p\right|\geq p((1-\nu_{n})\gamma_{n}-\nu_{n})  \right) \leq 2e^{-c\log^{3}(n)\gamma_{n}^{2}},
\end{equation}
for some $c>0$ whenever  $ 4 \nu_{n} \leq \gamma_{n} \leq 1$.
Using the inequality $\nu^{2}_{n} \leq 1/\log{n}$, and choosing $\beta^{2}_{n} = 1/\log{n}$, $\gamma^{2}_{n} = 16/\log{n}$ and $\zeta^{2}_{n} = 1/\log{n}$, we get the desired result by combining \eqref{eq:lower:hard} and \eqref{eq:lower:hard:2}.
\end{enumerate}
\subsection{Proof of Theorem \ref{thm:spectral}}
We begin by writing that
$$
\frac{1}{n}Y^{\top}Y = \frac{\|\bt\|^{2}}{n}\eta\eta^{\top} + Z_{1},
$$
where
$$
Z_{1} = \frac{1}{n}\eta\bt^{\top}W + \frac{1}{n}W^{\top}\bt\eta^{\top} + \frac{1}{n}W^{\top}W.
$$
Next observe that
$$
H\left(\frac{1}{n}Y^{\top}Y \right) = \frac{\|\bt\|^{2}}{n}\eta\eta^{\top} + Z_{2},
$$
where $Z_{2}$ is given by
$$ Z_{2}=  H\left( Z_{1}\right) - \frac{\|\bt\|^{2}}{n}\mathbf{I}_{n}. $$
Based on Lemma \ref{lem:diagonal}, we have
\begin{equation}\label{eq:useful}
\|Z_{2}\|_{op} \leq 4\left\|\frac{1}{n}\eta\bt^{\top}W \right\|_{op} + 2\left\|\frac{1}{n}W^{\top}W-\mathbf{E}\left(\frac{1}{n}W^{\top}W\right)\right\|_{op} + \frac{\|\bt\|^{2}}{n}.
\end{equation}
Using the Davis-Kahan $\sin{\theta}$ Theorem (cf .Theorem $4.5.5$ in \cite{vershinynHigh}), we obtain
\begin{equation}\label{eq:kahan}
\underset{\nu \in \{-1,1\}}{\min}\left\|\hat{v} - \frac{1}{\sqrt{n}} \nu\eta \right\|^{2} \leq 8 \frac{\|Z_{2}\|^{2}_{op}}{\|\bt\|^{4}}.
\end{equation}
Hence, using Lemma \ref{lem:round}, we get
\begin{equation}\label{eq:round:precise}
    \frac{1}{n}r(\eta^{0},\eta) \leq 16 \frac{\|Z_{2}\|^{2}_{op}}{\|\bt\|^{4}} \leq \frac{512}{\|\bt\|^{4}}\left( \left\|\frac{1}{n}\eta\bt^{\top}W \right\|_{op} + \left\|\frac{1}{n}W^{\top}W-\mathbf{E}\left(\frac{1}{n}W^{\top}W\right)\right\|_{op} \right)^2+\frac{32}{n^{2}}.
\end{equation}
Since $\mathbf{r}_{n}\geq C$, for some $C$ large enough. We may assume that $\|\bt\|^{2} \geq 1$ so that $1+p/n \leq \|\bt\|^{2}+p/n$. The inequality in expectation is a consequence of Lemma \ref{lem:outside} and Lemma \ref{lem:inside}.

\noindent For the inequality in probability, we first observe, using \eqref{eq:round:precise}, that
$$
\frac{1}{n}|\eta^{\top}\eta^{0}| \geq 1-8\frac{\|Z_{2}\|^{2}_{op}}{\|\bt\|^{4}}.
$$
Next, and since $\mathbf{r}_{n}\geq C$ for some $C$ large enough, observe that for $C' > 8$, we have
$$
\mathbf{P}_{(\bt,\eta)}\left( \frac{1}{n}|\eta^{\top}\eta^{0}| \leq 1 - \frac{\log{n}}{n} - \frac{C'}{\mathbf{r}_{n}^{2}}\right) \leq A_{1}+A_{2},
$$
where
$$A_{1}=
\mathbf{P}_{(\bt,\eta)}\left(\left\|\frac{1}{n}\eta \theta^{\top}W  \right\|_{op} \geq \sqrt{\frac{\log{n}}{2n}}\|\bt\|^{2} + 2\|\bt\| \right),
$$
and
$$
A_{2}=
\mathbf{P}_{(\bt,\eta)}\left(\left\|\frac{1}{n}W^{\top}W - \mathbf{E}\left( \frac{1}{n}W^{\top}W\right) \right\|_{op} \geq \sqrt{\frac{\log{n}}{2n}}\|\bt\|^{2} + C'\left(1 \vee \sqrt{p/n}\right) \right),
$$
Using Lemma \ref{lem:outside} and Lemma \ref{lem:inside}, we get
$$
\mathbf{P}_{(\bt,\eta)}\left( \frac{1}{n}|\eta^{\top}\eta^{0}| \leq 1 - \frac{\log{n}}{n} - \frac{C}{\mathbf{r}_{n}^{2}}\right) \leq 2e^{-c\sqrt{n\log{n}}\|\bt\|^{2}(1 \wedge \frac{\sqrt{n\log{n}}\|\bt\|^{2}}{p})} \leq e^{-c\sqrt{\log{n}}\mathbf{r}_{n}^{2}},
$$
Using the tail Gaussian function, we conclude easily that
$$
e^{-c\sqrt{\log{n}}\mathbf{r}_{n}^{2}} = o(\Phi^{{\sf c}}(\mathbf{r}_{n})).
$$

\subsection{Proof of Theorem \ref{thm:initial_lloyds}}
By the definition of $r(\hat{\eta},\eta)$, we may assume w.l.o.g that $\eta^{\top}\hat{\eta}^{0}>0$. Define the random events $\mathbf{A}_{i}$ for $i=1,\dots,n$, $\mathbf{B}$ and $\mathbf{C}$ such that for all $i=1,\dots,n$ 
$$
\mathbf{A}_{i} = \left\{  \left(\frac{1}{n}\mathbf{H}(Y^{\top}Y)_{i}^{\top}\eta\right)  \eta_{i} \geq \|\bt\|^{2} \left(\frac{8C}{\mathbf{r}_{n}} + \frac{C'}{\mathbf{r}_{n}^{2}} + 8c'\sqrt{\frac{\log{n}}{n}} + \nu_{n}\right) \right\},
$$
$$
\mathbf{C} = \left\{    \frac{1}{n}\sum_{i=1}^{n}\mathbf{1}_{\mathbf{A}_{i}} \leq \frac{C'}{4\mathbf{r}_{n}^{2}}\right\}
$$
and
$$
\mathbf{B} = \left\{ \|Z_{2}\|_{op} \leq c'\sqrt{\frac{\log{n}}{n}}\|\bt\|^{2} + C\left(1 \vee \sqrt{p/n}\right) + 2 \|\bt\|\right\},
$$
where we use the same notation of the previous proof and $c'$ a positive constant that we may choose large enough. 

We first prove, by induction, that on the event $\mathbf{B} \cap \mathbf{C}$, we have
$$
\frac{1}{n}\eta^{\top}\hat{\eta}^{k} \geq1- \frac{C'}{\mathbf{r}_{n}^{2}}-\nu_{n}, \quad \forall k =0,1,\dots
$$
For $k=0$, the result is obvious. Let $k \geq 1$. Assume that the result holds for $k$, and we prove it for $k+1$. Remember that 
$$
\frac{1}{n}\mathbf{H}(Y^{\top}Y) =\frac{1}{n}  \|\bt\|^{2}\eta\eta^{\top}  + Z_{2}.
$$
A simple calculation leads to
$$
 \frac{1}{n}\mathbf{H}(Y^{\top}Y)_{i}^{\top}\hat{\eta}^{k}  =  (Z_{2})_{i}^{\top}(\hat{\eta}^{k}-\eta) + \frac{1}{n}\mathbf{H}(Y^{\top}Y)_{i}^{\top}\eta - \|\bt\|^{2}\eta_{i}\frac{n- \eta ^{\top}\hat{\eta}^{k}}{n}.
$$
Hence if $\eta_{i} = -1$ and if $\mathbf{A}_{i}$ is true, then using the induction hypothesis we get
$$
 \frac{1}{n}\mathbf{H}(Y^{\top}Y)_{i}^{\top}\hat{\eta}^{k} \leq (Z_{2})_{i}^{\top}(\hat{\eta}^{k}-\eta ) -\|\bt\|^{2}\left( \frac{8C}{\mathbf{r}_{n}} +8c' \sqrt{\frac{\log{n}}{n}}\right).
$$
Hence when $\eta_{i} = -1$ we have
$$
\mathbf{1}_{\left\{  \frac{1}{n}\mathbf{H}(Y^{\top}Y)_{i}^{\top}\hat{\eta}^{k}  \geq 0  \right\} }\mathbf{1}_{\mathbf{A}_{i}} \leq \mathbf{1}_{ \left\{ (Z_{2})_{i}^{\top}(\hat{\eta}^{k}-\eta ) \geq \|\bt\|^{2}\left( \frac{8C}{\mathbf{r}_{n}} +8c' \sqrt{\frac{\log{n}}{n}}\right) \right\} } \leq \left(\frac{ (Z_{2})_{i}^{\top}(\hat{\eta}^{k}-\eta )}{\|\bt\|^{2}\left( \frac{8C}{\mathbf{r}_{n}} +8c' \sqrt{\frac{\log{n}}{n}}\right)  }\right)^{2}.
$$
similarly we get for $\eta_{i} = 1$ that 
$$
\mathbf{1}_{\left\{  \frac{1}{n}\mathbf{H}(Y^{\top}Y)_{i}^{\top}\hat{\eta}^{k}  \leq 0  \right\} }\mathbf{1}_{\mathbf{A}_{i}} \leq \left(\frac{ (Z_{2})_{i}^{\top}(\hat{\eta}^{k}-\eta )}{\|\bt\|^{2}\left( \frac{8C}{\mathbf{r}_{n}} +8c' \sqrt{\frac{\log{n}}{n}}\right)  }\right)^{2}.
$$
It is clear that
$$
\frac{1}{2}|\hat{\eta}^{k+1} - \eta| = \sum_{\eta_{i}=-1}\mathbf{1}_{\left\{  \frac{1}{n}\mathbf{H}(Y^{\top}Y)_{i}^{\top}\hat{\eta}^{k}  \geq 0  \right\} }  + \sum_{\eta_{i}=1}\mathbf{1}_{\left\{  \frac{1}{n}\mathbf{H}(Y^{\top}Y)_{i}^{\top}\hat{\eta}^{k}  \leq 0  \right\} }  .
$$
Hence we get using the events $\mathbf{A}_{i}$ for $i=1,\dots,n$, that 
\begin{equation}\label{eq:rec}
\frac{1}{2n}|\hat{\eta}^{k+1} - \eta| \leq \frac{ \|Z_{2}\|^{2}_{op}}{\|\bt\|^{4}\left( \frac{8C}{\mathbf{r}_{n}} +8c' \sqrt{\frac{\log{n}}{n}}\right)^{2}  } \frac{\|\hat{\eta}^{k}-\eta \|^{2}}{n} +  \frac{1}{n}\sum_{i=1}^{n}\mathbf{1}_{\mathbf{A}^{{\sf c}}_{i}}.
\end{equation}
Using the induction hypothesis and the events $\mathbf{B}$ and $\mathbf{C}$, we get
$$
1 - \frac{1}{n}\eta^{\top}\hat{\eta}^{k+1} \leq 4  \left(\frac{c'\sqrt{\frac{\log{n}}{n}}\|\bt\|^{2} + C\left(1 \vee \sqrt{p/n}\right)+ 2\|\bt\|}{\|\bt\|^{2}\left( \frac{8C}{\mathbf{r}_{n}} +8c' \sqrt{\frac{\log{n}}{n}}\right)}  \right)^{2} (C'/\mathbf{r}_{n}^{2}+\nu_{n}) + \frac{C'}{2\mathbf{r}_{n}^{2}}.
$$
Since $\mathbf{r}_{n}>C$ for $C$ large enough, then $(1\vee \sqrt{p/n}) \leq \| \bt \|^{2}/\mathbf{r}_{n}$ and $2\|\bt\| \leq \frac{\|\bt\|^2}{\mathbf{r}_n}$. It comes that
$$
\frac{1}{n}\eta^{\top}\hat{\eta}^{k+1} \geq 1 - \frac{C'}{\mathbf{r}_{n}^{2}} - \nu_{n}. 
$$
That concludes that on $\mathbf{B}\cap \mathbf{C}$, for all $k = 0,1,\dots$ we get
$$
\frac{1}{n}\eta^{\top}\hat{\eta}^{k} \geq 1 - \frac{C'}{\mathbf{r}_{n}^{2}} - \nu_{n}. 
$$
Hence, and using \eqref{eq:rec}, we obtain 
$$
\frac{1}{n}|\hat{\eta}^{k+1} - \eta|\mathbf{1}_{\mathbf{B}}\mathbf{1}_{\mathbf{C}} \leq \frac{1}{4}\frac{1}{n}|\hat{\eta}^{k}-\eta|\mathbf{1}_{\mathbf{B}}\mathbf{1}_{\mathbf{C}} + \frac{2}{n}\sum_{i=1}^{n}\mathbf{1}_{\mathbf{A}^{{\sf c}}_{i}}.
$$
As a consequence we find that for $k=0,1,\dots$
$$
\frac{1}{n}|\hat{\eta}^{k} - \eta|\mathbf{1}_{\mathbf{B}}\mathbf{1}_{\mathbf{C}} \leq 2 \left(\frac{1}{4}\right)^{k} + \frac{8}{3n}\sum_{i=1}^{n}\mathbf{1}_{\mathbf{A}^{{\sf c}}_{i}}.
$$
Observe that for $k \geq \lfloor 3\log{n} \rfloor$, we have $k \geq 2\frac{\log{n}}{ \log{4}}$ and 
$$
\left(\frac{1}{4}\right)^{k} \leq \frac{1}{n^{2}}.
$$
Hence for $k \geq \lfloor 3\log{n} \rfloor$,
$$
\frac{1}{n}|\hat{\eta}^{k} - \eta|\mathbf{1}_{\mathbf{B}}\mathbf{1}_{\mathbf{C}} \leq \frac{2}{n^{2}} + \frac{8}{3n}\sum_{i=1}^{n}\mathbf{1}_{\mathbf{A}^{{\sf c}}_{i}}.
$$
When $ \frac{1}{n}\sum_{i=1}^{n}\mathbf{1}_{\mathbf{A}^{{\sf c}}_{i}} = 0  $ then $ \frac{1}{n}|\hat{\eta}^{k} - \eta|\mathbf{1}_{\mathbf{B}}\mathbf{1}_{\mathbf{C}}=0$. In the opposite case, $ \frac{1}{n}\sum_{i=1}^{n}\mathbf{1}_{\mathbf{A}^{{\sf c}}_{i}} \geq \frac{1}{n}$. This leads to
$$
\frac{1}{n}|\hat{\eta}^{k} - \eta|\mathbf{1}_{\mathbf{B}}\mathbf{1}_{\mathbf{C}} \leq  \frac{14}{3n}\sum_{i=1}^{n}\mathbf{1}_{\mathbf{A}^{{\sf c}}_{i}}.
$$
Finally we get for $k\geq \lfloor3\log{n} \rfloor$,
$$
\frac{1}{n}\mathbf{E}\left(|\hat{\eta}^{k} - \eta| \right)\leq  \frac{14}{3n}\sum_{i=1}^{n}\mathbf{P}\left( \mathbf{A}^{{\sf c}}_{i} \right) + \mathbf{P}\left(\mathbf{B}^{{\sf c}} \right) + \mathbf{P}\left(\mathbf{C}^{{\sf c}} \right) \leq \left( \frac{14}{3} + \frac{4\mathbf{r}_{n}^{2}}{C'} \right)\frac{1}{n}\sum_{i=1}^{n}\mathbf{P}\left( \mathbf{A}^{{\sf c}}_{i} \right) + \mathbf{P}\left(\mathbf{B}^{{\sf c}} \right)  .
$$
The term $\mathbf{P}\left(\mathbf{B}^{{\sf c}} \right)$ is upper bounded exactly as in the proof of Theorem \ref{thm:spectral} and we have
$$
\mathbf{P}(\mathbf{B}^{{\sf c}}) = o(\Phi^{{\sf c}}(\mathbf{r}_{n})).
$$
For the other term observe that
$$
\mathbf{P}\left( \mathbf{A}^{{\sf c}}_{i} \right) \leq G_{\sigma}\left(\frac{C''}{\mathbf{r}_{n}}+\epsilon_{n},\|\bt\|^{2}\right),
$$
for some $C''>0$ and $\epsilon_{n}=o(1)$. That concludes the proof.

\subsection{Proof of Theorem \ref{thm:spectral_lloyds}}
Combining Theorem \ref{thm:spectral} and Theorem \ref{thm:initial_lloyds}, it is enough to prove that 
$$
\mathbf{r}_{n}^{2}\underset{\|\bt\| \geq \Delta }{\sup}G_{\sigma}\left( \epsilon_{n}+ \frac{C'}{\mathbf{r}_{n}},\bt\right) \leq \Phi^{{\sf c}}\left(\mathbf{r}_{n}\left(1-\epsilon'_{n}-\frac{C''\log{\mathbf{r}_{n}}}{\mathbf{r}_{n}} \right) \right)+\epsilon'_{n}\Phi^{{\sf c}}(\mathbf{r}_{n}),
$$
for some $\epsilon'_{n} = o(1)$ and $C''>0$.
Recall that
$$ G_{\sigma}\left(\epsilon_{n}+\frac{C'}{\mathbf{r}_{n}},\bt\right) = \mathbf{P}\left( (\bt + \xi_{1})^{\top}\left( \bt + \frac{\xi_{2}}{\sqrt{n-1}}\right) \leq \left( \epsilon_{n} + \frac{C'}{\mathbf{r}_{n}}\right)\|\bt\|^{2}  \right),  $$
where $\xi_{1},\xi_{2}$ are two independent Gaussian random vector with i.i.d. standard entries and $\bt$ and independent Gaussian prior. Moreover, using independence, we have
$$ G_{\sigma}\left(\epsilon_{n}+\frac{C'}{\mathbf{r}_{n}} ,\bt\right) = \mathbf{P}\left( \varepsilon\sqrt{\|\bt\|^{2} + \frac{\|\xi_{2}\|^{2}}{n-1} + \frac{2}{\sqrt{n-1}} \bt^{\top}\xi_{2}} \geq  \|\bt\|^{2}\left(1 - \epsilon_{n}-\frac{C'}{\mathbf{r}_{n}} \right) + \frac{1}{\sqrt{n-1}} \bt^{\top}\xi_{2} \right),  $$
where $\varepsilon$ is a standard Gaussian random variable. Set the random event
$$
\mathcal{A} = \left\{ \frac{\|\xi_{2}\|^{2}}{n-1} \leq \frac{p}{n-1}+\zeta_{n}\|\bt\|^{2} \right\}\cap\left\{ |\bt^{\top}\xi_{2}| \leq \sqrt{n-1}\beta_{n}\|\bt\|^{2} \right\},
$$
where $\zeta_{n}$ and $\beta_{n}$ are positive sequences.
It is easy to check that 
$$
\mathbf{P}\left( \mathcal{A}^{{\sf c}}\right) \leq e^{-c\|\bt\|^{4}n^{2}\zeta_{n}^{2}/p  } +2 e^{-c\beta_{n}^{2}n\|\bt\|^{2}} + e^{-c\zeta_{n}n\|\bt\|^{2}},
$$
for some $c>0$. Hence using the event $\mathcal{A}$, we get
$$
G_{\sigma}\left(\epsilon_{n}+\frac{C'}{\mathbf{r}_{n}} ,\bt\right) \leq  \mathbf{P}\left( \varepsilon\sqrt{\|\bt\|^{2}(1+\zeta_{n}+2\beta_{n}) + \frac{p}{n-1}} \geq  \|\bt\|^{2}\left(1 - \epsilon_{n}-\frac{C'}{\mathbf{r}_{n}}-\beta_{n} \right) \right) + \mathbf{P}(\mathcal{A}^{{\sf c}}).
$$
By choosing $\beta_{n} = \zeta_{n} = \sqrt{\frac{\log{n}}{n}}$, we get that 
$$
\mathbf{P}(\mathcal{A}^{{\sf c}}) \leq 4 e^{-c\sqrt{\log{n}}\mathbf{r}_{n}}.
$$
The last fact is due to the condition $\mathbf{r}_{n}\geq C$ for some $C>0$. Hence 
$$
\mathbf{P}(\mathcal{A}^{{\sf c}}) = o(\Phi^{{\sf c}}(\mathbf{r}_{n})).
$$
Moreover and since $\zeta_{n}$ and $\beta_{n}$ are vanishing sequences as $n \to \infty$, we get that
\begin{align*}
\mathbf{P}&\left( \varepsilon\sqrt{\|\bt\|^{2}(1+\zeta_{n}+2\beta_{n}) + \frac{p}{n-1}} \geq  \|\bt\|^{2}\left(1 - \epsilon_{n}-\frac{C'}{\mathbf{r}_{n}}-\beta_{n} \right) \right) =\\
&\Phi^{{\sf c}}\left( \frac{\|\bt\|^{2}}{\sqrt{\|\bt\|^2+\frac{p}{n}}}\left( 1 - \frac{C'}{\mathbf{r}_{n}}-\epsilon'_{n}\right) \right),
\end{align*}
for some $\epsilon'_{n} = o(1)$. We conclude using the fact that $x \to \frac{x}{\sqrt{x+\frac{p}{n}}}$ is non-decreasing on $\mathbf{R}^{+}$ and the fact that for $C<x<y$, we have  $x^{2}\Phi^{{\sf c}}(y)\leq c_{1}\Phi^{{\sf c}}(y - c_{2} \log{x}) $, for some $c_1,c_2>0$.

\subsection{Proof of Proposition \ref{prop:match_lower}}
Set $n$ large enough. According to Theorem \ref{thm:lower_asymptotic}, we have
\begin{equation}\label{eq:test2}
\Psi_{\Delta} \geq \frac{1}{2}\Phi^{{\sf c}}(2\mathbf{r}_{n}).
\end{equation}
For the upper bound. If $\mathbf{r}_{n}$ is larger than $2C$, then using Theorem \ref{thm:spectral_lloyds}, we get
\begin{equation}\label{eq:test1}
\Psi_{\Delta} \leq C' \Phi^{{\sf c}}\left(\frac{\mathbf{r}_{n}}{4}\right),
\end{equation}
for some $C'>0$. Observe that for $\mathbf{r}_{n}\leq 2C$, we have
$$
   c_{1} \leq \Phi^{{\sf c}}(\mathbf{r}_{n}),
$$
for some $c_{1}>0$. Hence, for $\mathbf{r}_{n}\leq 2C$, we get
\begin{equation}\label{eq:test3}
\Psi_{\Delta} \leq \frac{\Phi^{{\sf c}}(\mathbf{r}_{n})}{c_{1}}.
\end{equation}
We conclude combining \eqref{eq:test1}, \eqref{eq:test2} and \eqref{eq:test3}.
\subsection{Proof of Theorem \ref{thm:asymp1}}
\begin{itemize}
    \item \textit{Necessary conditions:}

According to Theorem \ref{thm:lower_asymptotic}, we have
$$
\Psi_{\Delta} \geq C'\Phi^{{\sf c}}(\mathbf{r}_{n}(1+\epsilon_{n})),
$$
for some $C'>0$ and $\epsilon_{n} = o(1)$.
If for some $\epsilon>0$, 
$$
\underset{n \to \infty}{\lim \inf}  n\Phi^{{\sf c}}(\mathbf{r}_{n}(1+\epsilon))>0.
$$
Hence, using the monotonicity of $\Phi^{{\sf c}}(.)$, we conclude that exact recovery is impossible. 

For Almost full recovery, assume that $\Phi^{{\sf c}}(\mathbf{r}_{n})$ does not converge to $0$, and that almost full recovery is possible. Then
using continuity and monotonicity of $\Phi^{{\sf c}}(.)$, we get that $ \mathbf{r}_{n}(1+\epsilon_{n}) \to \infty$. Hence $\mathbf{r}_{n}\to \infty$ and $\Phi^{{\sf c}}(\mathbf{r}_{n}) \to 0$ which is absurd. That proves that the condition $\Phi^{{\sf c}}(\mathbf{r}_{n}) \to 0$ is necessary to achieve almost full recovery.
\item \textit{Sufficient conditions:}

According to Theorem \ref{thm:spectral_lloyds}, we have that, under the condition $\mathbf{r}_{n}>C$ for some $C>0$, the estimator $\hat{\eta}^{k}$ defined in the Theorem satisfies
$$
\underset{(\bt,\eta)\in \Omega_{\Delta}}{\sup}\frac{1}{n}\mathbf{E}_{(\bt,\eta)}r(\eta^{k},\eta) \leq C'\Phi^{{\sf c}}\left(\mathbf{r}_{n}\left(1-\epsilon_{n}-\frac{C'\log{\mathbf{r}_{n}}}{\mathbf{r}_{n}}\right)\right),
$$
for some sequence $\epsilon_{n}$ such that $\epsilon_{n} = o(1)$. If $\Phi^{{\sf c}}(\mathbf{r}_{n})\to 0$, then $\mathbf{r}_{n} \to \infty$. Hence for any $\epsilon>0$, $\mathbf{r}_{n}(1-\epsilon) \to \infty$. It follows that $\mathbf{r}_{n}\left(1 - \epsilon_{n} - \frac{C'\log{\mathbf{r}_{n}}}{\mathbf{r}_{n}}\right) \to \infty$. We conclude that almost full recovery is possible under the condition $\Phi^{{\sf c}}(\mathbf{r}_{n}) \to 0$, and $\hat{\eta}^{k}$ achieves almost full recovery in that case.

For exact recovery, observe that, if 
$$
n\Phi^{{\sf c}}(\mathbf{r}_{n}(1-\epsilon)) \to 0,
$$
for some $\epsilon>0$, then $\mathbf{r}_{n} \to \infty$. It follows that for $n$ large enough
$$
          \mathbf{r}_{n}\left(1 - \epsilon_{n} - \frac{C'\log{\mathbf{r}_{n}}}{\mathbf{r}_{n}}\right) \geq \mathbf{r}_{n}(1-\epsilon).
$$
We conclude by taking the limit that $\hat{\eta}^{k}$ achieves exact recovery in that case, and that exact recovery is possible.
    
\end{itemize}

\subsection{Proof of Theorem \ref{thm:asymp} and \ref{thm:asympE}}
By inverting the function $x \to \frac{x}{\sqrt{x+\frac{p}{n}}}$, we observe that for any $A>0$,
$$
\mathbf{r}_{n}^{2} \geq A \quad \Leftrightarrow \quad \Delta_{n}^{2} \geq A\frac{1+\sqrt{1+\frac{4p}{nA}}}{2}.
$$
Using Theorem \ref{thm:asymp1} and the Gaussian tail function, we get immediately the results for both almost full recovery and exact recovery.

\section{Technical Lemmas}\label{sec:lem}
\begin{lemma}\label{lem:H}
Let $A$ be a matrix in $\mathbf{R}^{n \times n}$. Then
$$
 \| H(A)\|_{op} \leq 2 \| A \|_{op}.
$$
\end{lemma}
\begin{proof}
From the linearity of $H$, we have that
$$
 \| H(A)\|_{op} \leq  \| A\|_{op} + \| \text{diag}(A)\|_{op},
$$
where
$$
\| \text{diag}(A)\|_{op} = \underset{i}{\max}|A_{ii}| \leq \|A\|_{op}.
$$
\end{proof}
\begin{lemma}\label{lem:diagonal}
For any random matrix $W$ with independent columns, we have
$$
\| H(W^{\top}W) \|_{op} \leq 2\left\| W^{\top}W - \mathbf{E}\left( W^{\top}W\right) \right\|_{op}.
$$
\end{lemma}
\begin{proof}
Since $\mathbf{E}\left( W^{\top}W\right)$ is a diagonal matrix, it follows that
$$ H(W^{\top}W) = H\left(W^{\top}W -\mathbf{E}\left(W^{\top}W\right)\right).$$
The result follows from Lemma \ref{lem:H}.
\end{proof}
\begin{lemma}\label{lem:outside}
Let $u \in \mathbf{S}^{p-1}$ and $v \in \mathbf{S}^{n-1}$, and $W\in \mathbf{R}^{p\times n}$ a matrix with i.i.d. centered Gaussian entries of variance at most $\sigma^{2}$. Then, for some $c,C>0$
$$ \forall t \geq 2\sigma, \quad \mathbf{P}\left(  \left\| \frac{1}{\sqrt{n}}W^{\top}uv^{\top} \right\|_{op} \geq t \right) \leq e^{-cnt/\sigma},
$$ 
and 
$$
\mathbf{E}\left(  \left\| \frac{1}{\sqrt{n}}W^{\top}uv^{\top} \right\|^{2}_{op}  \right) \leq C\sigma^{2}.
$$
\end{lemma}
\begin{proof}
We can easily check that
$$
\left\| \frac{1}{\sqrt{n}}W^{\top}uv^{\top} \right\|_{op} \leq \frac{1}{\sqrt{n}} \|W^{\top}u\|_{2}.
$$
Since $\|u\|_{2}=1$, we have that $W^{\top}u$ is Gaussian with mean $0$ and covariance matrix $\sigma^{2}\mathbf{I}_{n}$. We conclude using a tail inequality for quadratic forms of sub-Gaussian random variables using the fact that $t \geq 2\sigma$, see, e.g., \cite{hsutail}. The inequality in expectation is immediate by integration of the tail function.
\end{proof}

\begin{lemma}\label{lem:inside}
Let $W\in \mathbf{R}^{p\times n}$ be a matrix with i.i.d. centered Gaussian entries of variance at most $\sigma^{2}$. For some $c,C,C'>0$ we have
$$
\forall t \geq C\sigma^{2} \left(1 \vee \sqrt{\frac{p}{n}}\right), \quad \mathbf{P}\left( \frac{1}{n}\| H(W^{\top}W)\|_{op} \geq t\right)  \leq e^{-cnt/\sigma^{2}\left(1 \wedge \frac{tn}{p\sigma^{2}}\right)},
$$
and 
$$
 \mathbf{E}\left( \frac{1}{n}\| H(W^{\top}W)\|^{2}_{op}\right)  \leq C'\sigma^{4}(1+p/n).
$$
\end{lemma}
\begin{proof}
Using Lemma \ref{lem:diagonal}, we get
$$
\mathbf{P}\left( \frac{1}{n}\| H(W^{\top}W)\|_{op} \geq t\right)  \leq \mathbf{P}\left( \frac{1}{n}\| W^{\top}W - \mathbf{E}(W^{\top}W)\|_{op} \geq t/2\right).
$$
Now based on Theorem $4.6.1$ in \cite{vershinynHigh}, we get moreover that
$$
\mathbf{P}\left( \frac{1}{n}\| H(W^{\top}W)\|_{op} \geq t\right)  \leq 9^{n}2e^{-cnt/\sigma^2(1\wedge tn/(p\sigma^2))},
$$
for some $c>0$. For $t \geq C(1\vee \sqrt{p/n})\sigma^{2}$ with $C$ large enough, we get $ ct(1\wedge tn/(p\sigma^{2})) \geq 4\sigma^{2}\log{9} $, hence
$$
\mathbf{P}\left( \frac{1}{n}\| H(W^{\top}W)\|_{op} \geq t\right)  \leq e^{-c'nt/\sigma^{2}(1\wedge tn/(p\sigma^{2}))},
$$
for some $c'>0$. The result in expectation is immediate by integration.
\end{proof}
\begin{lemma}\label{lem:round}
For any $x \in \{-1,1\}^{n}$ and $y\in \mathbf{R}^{n}$, we have
$$
\frac{1}{n}\left| x - \text{sign}(y) \right| \leq 2\left\| \frac{x}{\sqrt{n}} - y \right\|^{2}.
$$
\end{lemma}
\begin{proof}
It is enough to observe that if $x_{i} \in \{-1,1\}$, then
$$
|x_{i} - \text{sign}(y_{i})| = 2\mathbf{1}(x_{i} \neq \text{sign}(y_{i})) \leq 2|x_{i}-\sqrt{n}y_{i}|^{2}.
$$
\end{proof}

\end{document}